\documentclass[a4paper, english]{article}

\author{Richard Ueltzen
    \thanks{University of Bonn, 53113 Bonn, Germany. Email address: \texttt{richard.ueltzen@gmail.com}\\
        The author is supported by a scholarship from the German Academic Scholarship Foundation (Studienstiftung des deutschen Volkes).}}
\date{}
\title{Characterizing graphs with high inducibility}

\usepackage[margin=2.5cm]{geometry}

\usepackage[english]{babel}

\usepackage{amsmath}
\usepackage{amssymb}
\usepackage{amsthm}
\usepackage{environ}
\usepackage{appendix}

\usepackage{mathtools}

\usepackage[colorinlistoftodos]{todonotes} 
\usepackage{microtype}

\usepackage{physics}
\usepackage{bbm}
\usepackage{ifthen}
\usepackage{xparse}

\newcommand{\ind}{\mathrm{ind}}

\newcommand{\Geom}{\mathrm{Geom}}
\newcommand{\Binom}{\mathrm{Bin}}

\newcommand{\A}{\mathcal A}
\newcommand{\B}{\mathcal B}
\newcommand{\D}{\mathcal D}
\newcommand{\E}{\mathcal E}
\newcommand{\F}{\mathcal F}
\newcommand{\IR}{\mathbb R}
\newcommand{\IZ}{\mathbb Z}
\newcommand{\IN}{\mathbb N}
\newcommand{\IP}{\mathbb P}
\newcommand{\IE}{\mathbb E}
\newcommand{\one}{\mathbbm 1}
\renewcommand{\epsilon}{\varepsilon}
\newcommand{\dcup}{\,\dot\cup\,}
\newcommand{\mmid}{\,\middle|\,}
\newcommand{\indd}[2]{I({#1}, #2)}

\newcommand{\brightness}[1]{\nu(#1)}
\DeclareMathOperator{\aut}{aut}

\NewDocumentCommand{\degree}{O{} O{} O{} m}{
    \ifthenelse{\equal{#3}{}} {
        \ifthenelse{\equal{#2}{}}{
            \ifthenelse{\equal{#1}{}}{
                d(#4)
            }{
                d(#4, #1)
            }
        }{
            \ifthenelse{\equal{#1}{}}{
                d_{#2}(#4)
            }{
                d_{#2}(#4, #1)
            }
        }
    }{
        \ifthenelse{\equal{#2}{}}{
            \ifthenelse{\equal{#1}{}}{
                d^{#3}(#4)
            }{
                d^{#3}(#4, #1)
            }
        }{
            \ifthenelse{\equal{#1}{}}{
                d^{#3}_{#2}(#4)
            }{
                d^{#3}_{#2}(#4, #1)
            }
        }
    }
}
\NewDocumentCommand{\neighbors}{O{} O{} O{} m}{
    \ifthenelse{\equal{#3}{}} {
        \ifthenelse{\equal{#2}{}}{
            \ifthenelse{\equal{#1}{}}{
                N(#4)
            }{
                N(#4, #1)
            }
        }{
            \ifthenelse{\equal{#1}{}}{
                N_{#2}(#4)
            }{
                N_{#2}(#4, #1)
            }
        }
    }{
        \ifthenelse{\equal{#2}{}}{
            \ifthenelse{\equal{#1}{}}{
                N^{#3}(#4)
            }{
                N^{#3}(#4, #1)
            }
        }{
            \ifthenelse{\equal{#1}{}}{
                N^{#3}_{#2}(#4)
            }{
                N^{#3}_{#2}(#4, #1)
            }
        }
    }
}

\usepackage[T1]{fontenc}
\usepackage[utf8]{inputenc}
\usepackage{csquotes}
\usepackage[hidelinks]{hyperref}
\usepackage[style=numeric, sorting=nyt, sortcites=true, giveninits=true, maxbibnames=99, doi=true, backend=biber]{biblatex}
\usepackage[english,capitalize,noabbrev]{cleveref}
\usepackage{bookmark}
\hypersetup{unicode=true}

\usepackage{etoolbox}
\AtBeginEnvironment{appendices}{\crefalias{section}{appendix}}

\AtEveryBibitem{
    \clearfield{note}  
    \clearfield{url}   
    \clearfield{issn}
    \clearfield{isbn}
}
\AtEveryBibitem{
    \ifentrytype{book}{%
        \clearfield{pages} 
    }{}
}

\renewbibmacro{in:}{}
\DeclareFieldFormat*{title}{\mkbibitalic{#1}}
\DeclareFieldFormat*{journaltitle}{#1}
\DeclareFieldFormat*{pages}{#1}

\AtEveryBibitem{\clearfield{number}}

\usetikzlibrary{calc}
\usetikzlibrary{arrows}

\tikzset{
    vertex/.style={
            circle,
            draw,          
            fill=none,
            minimum size=6pt, 
            inner sep=0pt,  
            line width=0.5pt
        },
    edge/.style={
            draw,
            line width=0.5pt
        },
    box/.style={
            rounded corners=10pt, 
        },
    highlight/.style={
            fill opacity=0.2,    
            draw opacity=0.7,    
            line width=1pt,      
            fill=#1,
            color=#1
        }
}

\NewEnviron{killcontents}{}
\theoremstyle{plain}
\newtheorem{theorem}{Theorem}

\newtheorem{lemma}[theorem]{Lemma}
\newtheorem{corollary}[theorem]{Corollary}
\newtheorem{claim}[theorem]{Claim}
\newtheorem{fact}[theorem]{Fact}
\newtheorem{conjecture}[theorem]{Conjecture}
\theoremstyle{remark}
\newtheorem{remark}[theorem]{Remark}
\theoremstyle{definition}
\newtheorem{definition}[theorem]{Definition}
\numberwithin{equation}{section}
\numberwithin{theorem}{section}

\crefname{claim}{Claim}{Claims}
\Crefname{equation}{}{}

\makeatletter
\def\namedlabel#1#2{\begingroup
    #2%
    \def\@currentlabel{#2}%
    \phantomsection\label{#1}\endgroup
}

\begin{document}

\maketitle

\begin{abstract}
    For a positive integer $k$ and a graph $H$ on $k$ vertices, we are interested in the inducibility of $H$, denoted $\mathrm{ind}(H)$, which is defined as the maximum possible probability that choosing $k$ vertices uniformly at random from a large graph $G$, they induce a copy of $H$.
    It follows from the resolved Edge-statistics conjecture that if $H \not \in \{K_k, \bar K_k\}$, then $\mathrm{ind}(H) \leq 1 / e + o_k(1)$.
    Equality holds for the star graph $K_{1, k-1}$, the graph with a single edge on $k$ vertices and their complements.
    We prove that for all other graphs $H$, we have $\mathrm{ind}(H) \leq c + o_k(1)$ for an absolute constant $c < 1 / e$.
    Moreover, we explicitly characterize all graphs with inducibility bounded away from zero.
    Namely, we show that this is the class of graphs $H$ for which there is a set $V_0 \subseteq V(H)$ of bounded size with the property that all permutations of $V(H) \backslash V_0$ extend to an automorphism of $H$.
\end{abstract}

\section{Introduction}\label{section:introduction}
Consider a positive integer $k$ and a graph $H$ on $k$ vertices.
We are interested in the maximum possible probability that a uniformly randomly chosen $k$-vertex set in a large graph $G$ induces a copy of $H$.
This quantity is the inducibility of $H$, which was first introduced by Pippenger and Golumbic~\cite{pippenger-golumbic:1975}.
For a more precise definition, suppose that $n \geq k$ is a positive integer, and $G$ is a graph on $n$ vertices.
Taking $W$ to be a uniformly random set of $k$ vertices of $G$, we define the \textit{induced density} of $H$ in $G$ as $\indd{H}{G} \coloneqq \IP[G[W] \cong H]$, the probability that $W$ induces a copy of $H$ in $G$.
We denote the maximum of $\indd{H}{G}$ over all $n$-vertex graphs $G$ by $\indd{H}{n}$.
By an averaging argument, $n \mapsto \indd{H}{n}$ is monotonically decreasing.
Hence, the \textit{inducibility} of $H$, defined as $\ind(H) \coloneqq \lim_{n \to \infty} \indd{H}{n} \in [0, 1]$, exists.

In recent years, the question of bounding the inducibility of graphs received much attention (see, for example, \cite{kral-norin-volec-2019,balogh-hu-lidicky-pfender:2016,yuster-2019,hefetz-tyomkyn-2018}).
However, the exact inducibility is only known for a few small graphs as well as for specific classes of graphs. For example, the inducibility of the cycle on six vertices is unknown.

In this paper, we give general asymptotic upper bounds for graph inducibilities.
This follows a direction of research proposed by Alon, Hefetz, Krivelevich and Tyomkyn~\cite[Conjecture 1.2]{alon-hefetz-krivelevich-tyomkyn:2018}, who conjectured that if $H$ is not the complete graph $K_k$ or its complement $\overline K_k$, then $\ind(H) \leq 1 / e + o_k(1)$.
Since inducibility is invariant under taking complements, i.e., $\ind(H) = \ind(\bar H)$, it is enough to consider the case where the number of edges of $H$ is $\ell \leq \frac{1}{2} \binom{k}{2}$.
Kwan, Sudakov, and Tran~\cite{kwan-sudakov-tran:2019} proved the conjecture in the case $Ck \leq \ell \leq \frac{1}{2} \binom{k}{2}$ for some absolute constant $C$.
Subsequently, the proof of the conjecture was independently completed by Fox and Sauermann~\cite{fox-sauermann:2020}, as well as by Martinsson, Mousset, Noever, and Truji\'c~\cite{martinsson-mousset-noever-trujic:2020}.
The papers above also resolved a stronger conjecture of Alon et al.~\cite[Conjecture 1.1]{alon-hefetz-krivelevich-tyomkyn:2018} on the number of edges induced by a uniformly random $k$-vertex set. In conclusion, the following theorem holds:
\begin{theorem}\label{theorem:edge-statistics-conjecture}
    For all positive integers $k$ and any graph $H$ on $k$ vertices with $H \not \in \{K_k, \overline K_k\}$, we have
    \begin{equation*}
        \ind(H) \leq \frac{1}{e} + o_k(1).
    \end{equation*}
\end{theorem}
As noted by Alon et al.~\cite{alon-hefetz-krivelevich-tyomkyn:2018}, this upper bound is tight, which the following two examples show:
Consider the complete bipartite graph $G$ with parts of sizes $n / k$ (red) and $n(k-1) / k$ (blue) on $n$ vertices with $k \mid n$.
A straightforward computation shows that when picking $k$ of its vertices uniformly at random, the probability of picking exactly one red vertex tends to $1 / e + o_k(1)$ as $n \to \infty$. Therefore, the inducibility of the star graph $K_{1, k-1}$ satisfies $\ind(K_{1, k-1}) \geq 1 / e + o_k(1)$. Moreover, in the random graph $G(n, p)$ on $n \geq k$ vertices with $p = \binom{k}{2}^{-1}$, the expected density of $k$-vertex sets inducing a single edge is $1 / e + o_k(1)$. Hence, the $k$-vertex graph $M_k^{(1)}$ with a single edge has inducibility $\ind(M_k^{(1)}) \geq 1 / e + o_k(1)$.

The research question addressed in this paper is the following: Are the above examples the only ones achieving equality in \cref{theorem:edge-statistics-conjecture}?
We answer this question affirmatively:

\begin{theorem}\label{theorem:main-result}
    There exists an absolute constant $c < 1 / e$ such that for all positive integers $k$ and $\ell$ with $2 \leq \ell \leq \frac{1}{2} \binom{k}{2}$ and any graph $H$ with $k$ vertices and $\ell$ edges satisfying $H \not\cong K_{1, k-1}$, we have
    \begin{equation*}
        \ind(H) \leq c + o_k(1).
    \end{equation*}
\end{theorem}

This shows that except for the known equality cases, the bound in \cref{theorem:edge-statistics-conjecture} can be strengthened to $c + o_k(1)$ for some absolute constant $c < 1 / e$.
Our second main result shows that all graphs with high inducibility have a specific structure that is given by the following definition (see \cref{figure:tamed-graph} for an illustration):

\begin{definition}
    For any positive integer $D$ and any graph $H$, we say that $H$ is \textit{$D$-tame} if there exists a subset $V_0 \subseteq V(H)$ of size at most $D$ such that any permutation $\pi$ of $V(H)$, with the property that $\pi(v) = v$ for every $v \in V_0$, is an automorphism of $H$. In this case, we say that $V_0$ tames $H$.
\end{definition}

Equivalently, $V_0 \subseteq V(H)$ tames $H$ if and only if the following two conditions hold
\begin{itemize}
    \item $V(H) \backslash V_0$ is either a clique or a stable set,
    \item for each $v \in V_0$, $v$ is either connected to no vertex in $V(H) \backslash V_0$ or to all vertices in $V(H) \backslash V_0$.
\end{itemize}

\begin{figure}[ht]
    \centering
    \includegraphics{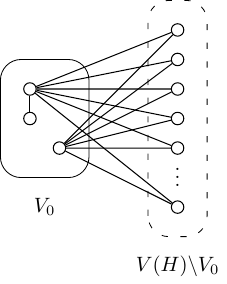}
    \caption{The graph $H$ is tamed by $V_0$.}
    \label{figure:tamed-graph}
\end{figure}

It is not hard to show, by adapting the construction of $G$ for $H = K_{1, k-1}$, that all $D$-tame graphs have inducibility at least $\gamma = \gamma(D) > 0$. Conversely, the following theorem shows that all graphs with inducibility bounded away from $0$ exhibit this structure:
\begin{theorem}\label{theorem:characterization-of-high-inducibility-graphs}
    For any fixed $\gamma > 0$, there exists a positive integer $D$ such that any graph $H$ satisfying $\ind(H) \geq \gamma$ is $D$-tame.
\end{theorem}
\begin{remark}
    The number of automorphisms of a $D$-tame graph $H$ with $k$ vertices satisfies $\aut(H) \geq (k-D)!$, and a converse also holds (see the author's bachelor's thesis~\cite[Appendix B]{ueltzen:bs-thesis}). The following equivalences summarize these results:
    \begin{equation*}
        \ind(H) \geq \Omega(1) \quad \iff \quad H~\text{is $O(1)$-tame} \quad \iff \quad
        \aut(H) \geq (k - O(1))!.
    \end{equation*}
\end{remark}

It is natural to ask how small we can choose the constant $c$ in \cref{theorem:main-result}.
Adjusting the constructions for the equality cases in \cref{theorem:edge-statistics-conjecture} shows that the complete bipartite graph with parts of sizes $2$ and $k-2$ and the graph with $k$ vertices and $2$ non-adjacent edges both have inducibility at least $2 / e^2 + o_k(1)$.
The same is true for the graph obtained from $K_{2, k-2}$ by adding an edge between the vertices in the part of size $2$.
Hence, the constant $c$ in \cref{theorem:main-result} is at least $2 / e^2$. We conjecture that this bound is tight:
\begin{conjecture}\label{conjecture:upper-bound-2-over-e-squared}
    For all positive integers $k$ and $\ell$ with $2 \leq \ell \leq \frac{1}{2} \binom{k}{2}$ and any graph $H$ with $k$ vertices and $\ell$ edges such that $H$ is not $K_{1, k-1}$, we have
    \begin{equation*}
        \ind(H) \leq \frac{2}{e^2} + o_k(1).
    \end{equation*}
\end{conjecture}

\section{Preliminaries and proof approach}

We use the following corollary of a result of Kwan, Sudakov and Tran~\cite[Theorem 1.1]{kwan-sudakov-tran:2019}:
\begin{theorem}\label{theorem:kwan-sudakov-tran-corollary}
    For any fixed $\gamma > 0$, there exists a fixed $C = C(\gamma)$ with the following property: For all positive integers $k$ and $\ell$ with $Ck \leq \ell \leq \frac{1}{2} \binom{k}{2}$ and any graph $H$ with $k$ vertices and $\ell$ edges, we have
    \begin{equation*}
        \ind(H) \leq \gamma.
    \end{equation*}
\end{theorem}
\cref{theorem:kwan-sudakov-tran-corollary} implies that, in the proofs of \cref{theorem:main-result,theorem:characterization-of-high-inducibility-graphs}, we only need to consider the case, where $H$ has $\ell \leq O_k(k)$ edges. Therefore, the following theorem shows that the bound in \cref{conjecture:upper-bound-2-over-e-squared} holds if the number of non-isolated vertices of $H$, denoted as $m(H)$, satisfies $m(H) \geq \Omega_k(k)$:
\begin{theorem}\label{theorem:main-result-for-denser-graphs}
    For any fixed $\alpha > 0$ and $C$, the following holds: For all positive integers $k$ and $\ell$ with $\ell \leq Ck$ and any graph $H$ with $k$ vertices and $\ell$ edges such that $H \not\cong K_{1,k-1}$ and $m(H) \geq \alpha k$, we have
    \begin{equation*}
        \ind(H) \leq \frac{2}{e^2} + o_k(1).
    \end{equation*}
\end{theorem}
In the case $m(H) \leq o_k(k)$, we only slightly improve on the bound in \cref{theorem:edge-statistics-conjecture}:
\begin{theorem}\label{theorem:main-result-for-sparser-graphs}
    There exist absolute constants $\alpha > 0$ and $c$ with $2 / e^2 \leq c < 1 / e$ such that for all positive integers $k$ and $\ell$ with $\ell \geq 2$ and any graph $H$ with $k$ vertices and $\ell$ edges satisfying $m(H) \leq \alpha k$, we have
    \begin{equation*}
        \ind(H) \leq c.
    \end{equation*}
\end{theorem}
\begin{proof}[Proof of \cref{theorem:main-result} assuming \cref{theorem:main-result-for-denser-graphs,theorem:main-result-for-sparser-graphs}]
    Let $\alpha$ and $c$ be the constants in \cref{theorem:main-result-for-sparser-graphs}. Moreover, let $C = C(c)$ be the constant in \cref{theorem:kwan-sudakov-tran-corollary}.
    Consider positive integers $k$ and $\ell$ with $2 \leq \ell \leq \frac{1}{2} \binom{k}{2}$ and a graph $H$ with $k$ vertices and $\ell$ edges satisfying $H \not\cong K_{1, k-1}$.
    It is sufficient to prove
    \begin{equation}\label{equation:desired-claim-for-theorem-main-result}
        \ind(H) \leq c + o_k(1),
    \end{equation}
    as $c < 1 / e$.
    If $\ell \geq C k$, then \cref{theorem:kwan-sudakov-tran-corollary} and $\ell \leq \frac{1}{2} \binom{k}{2}$ imply \cref{equation:desired-claim-for-theorem-main-result}.
    Moreover, if $m(H) \leq \alpha k$, then \cref{equation:desired-claim-for-theorem-main-result} follows from \cref{theorem:main-result-for-sparser-graphs} and $\ell \geq 2$.
    We now consider the remaining case that $\ell \leq C k$ and $m(H) \geq \alpha k$. Since $H \not\cong K_{1,k-1}$, \cref{theorem:main-result-for-denser-graphs} and $c \geq 2 / e^2$ imply
    \begin{equation*}
        \ind(H) \leq \frac{2}{e^2} + o_k(1) \leq c + o_k(1).\qedhere
    \end{equation*}
\end{proof}

We deduce \cref{theorem:characterization-of-high-inducibility-graphs} from the following, similar theorem combined with \cref{theorem:kwan-sudakov-tran-corollary}:
\begin{theorem}\label{theorem:weak-characterization-of-high-inducibility-graphs}
    For any fixed $\gamma > 0$ and $C$, there exists a constant $D = D(\gamma, C)$ with the following property: For all positive integers $k$ and $\ell$ with $\ell \leq Ck$ and any graph $H$ with $k$ vertices and $\ell$ edges such that $H$ is not $D$-tame, we have
    \begin{equation*}
        \ind(H) \leq \gamma.
    \end{equation*}
\end{theorem}
\begin{proof}[Proof of \cref{theorem:characterization-of-high-inducibility-graphs} assuming \cref{theorem:weak-characterization-of-high-inducibility-graphs}]
    For a fixed $\gamma > 0$, let $C \coloneqq C(\gamma)$ be the constant in \cref{theorem:kwan-sudakov-tran-corollary} and let $D \coloneqq D(\gamma, C)$ be the constant in \cref{theorem:weak-characterization-of-high-inducibility-graphs}. Consider positive integers $k$ and $\ell$ and a graph $H$ with $k$ vertices and $\ell$ edges such that $H$ is not $D$-tame.
    Since inducibility and $D$-tameness are invariant under taking complements, we may assume $\ell \leq \frac{1}{2} \binom{k}{2}$. If $Ck \leq \ell \leq \frac{1}{2} \binom{k}{2}$, \cref{theorem:kwan-sudakov-tran-corollary} implies $\ind(H) \leq \gamma$. In the remaining case, $\ell \leq Ck$, \cref{theorem:weak-characterization-of-high-inducibility-graphs} yields $\ind(H) \leq \gamma$.
\end{proof}

\subsection{Proof approach}

The rest of this paper contains the proofs of \cref{theorem:main-result-for-denser-graphs,theorem:main-result-for-sparser-graphs,theorem:weak-characterization-of-high-inducibility-graphs}.
\cref{section:proofs-of-theorem:main-result-for-denser-graphs-and-theorem:weak-characterization-of-high-inducibility-graphs} begins with three lemmas that, very roughly, provide upper bounds on the inducibility of a graph $H$ with $k$ vertices and $\ell$ edges, if for some fixed $C$ and $\epsilon > 0$,
\begin{enumerate}
    \item the number of edges in $H$ satisfies $\ell \leq Ck$ and $H$ has maximum degree at least $\epsilon k$,
    \item the maximum degree in $H$ is at most $\epsilon k$ and $(1 - o_k(1))k$ vertices in $H$ have degree $\tau$ for some $\tau \geq 1$,
    \item the number of edges in $H$ satisfies $\ell \leq Ck$ and there is no $\tau \geq 0$ such that $(1 - o_k(1))k$ vertices have degree $\tau$,
\end{enumerate}
respectively.
The first two lemmas are proven in \cref{section:proof-of-lemma:ind-if-high-degree-vtcs,section:proof-of-lemma:ind-if-almost-uniform-low-maximum-degree}.
The third lemma is a direct consequence of a result of Martinsson et al.~\cite[Claim 3.3]{martinsson-mousset-noever-trujic:2020}. An alternative proof can be found in the author's bachelor's thesis~\cite[Section 6]{ueltzen:bs-thesis}.
For any graph $H$ with $k$ vertices and $\ell \leq Ck$ edges, at least one of the three lemmas implies an upper bound on $\ind(H)$, unless $(1 - o_k(1))k$ vertices in $H$ have degree $\tau = 0$, i.e., $m(H) \leq o_k(k)$. This is why \cref{theorem:main-result-for-denser-graphs} follows from the lemmas, as we show in \cref{subsection:proof-of-theorem:main-result-for-denser-graphs}.
The proof of \cref{theorem:weak-characterization-of-high-inducibility-graphs} in \cref{subsection:proof-of-theorem:weak-characterization-of-high-inducibility-graphs} is similar.
A key difference is that we must address the case $m(H) \leq o_k(k)$, for which we will use the following corollary of two lemmas due to Fox and Sauermann~\cite[Lemmas 3.1 and 3.2]{fox-sauermann:2020}:
\begin{theorem}\label{theorem:fox-sauermann:2020}
    For any fixed $\gamma > 0$, there exists a fixed positive integer $D = D(\gamma)$ such that for all positive integers $k$ and any graph $H$ with $k$ vertices and $D \leq m(H) \leq k / 32$,
    \begin{equation*}
        \ind(H) \leq \gamma.
    \end{equation*}
\end{theorem}
\cref{theorem:main-result-for-sparser-graphs} is proven in \cref{section:proof-of-theorem:main-result-for-sparser-graphs} using a key lemma that we prove in \cref{section:proof-of-lemma:ind-if-sparse}. The proof strategy for the key lemma is a modified version of an argument by Martinsson et al.~\cite[Proposition 2.1]{martinsson-mousset-noever-trujic:2020}: We pick vertices $v_1, v_2 \dots$ of a large graph $G$ uniformly at random, independently and consecutively, and count the number $X$ of times 'something happens.' We will show that, conditioned on $G[\{v_1, \dots, v_k\}] \cong H$, $X$ is very concentrated. Due to the setup of the random experiment, we can determine the distribution of $X$ and show that it is not concentrated. In total, this yields an upper bound on the induced density of $H$ in $G$. \cref{section:proof-of-theorem:main-result-for-sparser-graphs} is concerned with developing the tools to define 'something happening' in such a way that, conditioned on $G[\{v_1, \dots, v_k\}] \cong H$, $X$ is indeed concentrated.

\subsection{Notation}\label{subsection:notation}

For a graph $G$, we denote the \textit{sets of vertices} and \textit{edges} of $G$ with $V(G)$ and $E(G)$, respectively.
For a vertex set $W \subseteq V(G)$, we denote the \textit{induced subgraph} of $G$ with vertices $W$ as $G[W]$.
Moreover, for $v \in V(G)$, the graph $G - v$ is defined as $G[V(G) \backslash \{v\}]$.
For $v \in V(G)$ and $W \subseteq V(G)$, we let $\neighbors[W][G]{v}$ be the \textit{set of neighbors} of $v$ in $W$ and let $\degree[W][G]{v} \coloneqq |\neighbors[W][G]{v}|$ be the \textit{degree} of $v$ into $W$.
We abbreviate $\neighbors[V(G)][G]{v}$ as $\neighbors[][G]{v}$ and $\degree[V(G)][G]{v}$ as $\degree[][G]{v}$.
We denote the \textit{maximum degree} of a vertex in $G$ by $\Delta(G)$.
Furthermore, we denote the number of non-isolated vertices in $G$ by $m(G)$ and define $M(G)$ as the induced subgraph of $G$ restricted to its non-isolated vertices.
We denote the \textit{complete graph} on $k$ vertices as $K_k$ and the \textit{complete bipartite graph} with parts of sizes $a$ and $b$ as $K_{a,b}$.
When there is no risk of confusion, we omit the subscript $G$ from all notation.

We denote the set of \textit{integers} with $\IZ$.
For a set $S$ and a nonnegative integer $f$, we denote the set of $f$-element subsets of $S$ by $\binom{S}{f}$.

The \textit{probability measure} in our probability space is always denoted by $\IP$.
For an event $\A$, we denote the complement of $\A$ as $\A^c$ and its indicator random variable as $\one_{\A}$.

In our asymptotic notation, we indicate the variable interpreted as going to infinity in the subscript.
For example, an expression is $o_k(1)$ if, for any fixed values of the constants, it approaches $0$ as $k \to \infty$, regardless of the values of all other variables.
In this context, a variable is treated as constant if it is specified as fixed in its definition or if it is listed after a '$\mid$' in the subscript.
For example, an expression is $o_{n \mid k}(1)$ if it approaches $0$ as $n \to \infty$ for any fixed value of $k$.

\section{Proofs of Theorems~\ref{theorem:main-result-for-denser-graphs}~and~\ref{theorem:weak-characterization-of-high-inducibility-graphs}}\label{section:proofs-of-theorem:main-result-for-denser-graphs-and-theorem:weak-characterization-of-high-inducibility-graphs}

The proofs of \cref{theorem:main-result-for-denser-graphs,theorem:weak-characterization-of-high-inducibility-graphs} are based on the following three lemmas, which we will prove in the following sections.

\begin{lemma}\label{lemma:ind-if-high-degree-vtcs}
    Fix a positive real number $\delta$ and positive integers $s$ and $t$. Consider a positive integer $k$ and a graph $H$ on $k$ vertices. Suppose that real numbers $a, b \in (0, 1)$ with $b - a \geq \delta$ are chosen such that no vertex in $H$ has degree in the interval $(ak, bk)$ and $S \coloneqq \{v \in V(H) \mid \degree{v} \geq bk\}$ has size $s$. Then, the following bounds on $\ind(H)$ hold:
    \begin{itemize}
        \item[\namedlabel{item:part1-lemma:ind-if-high-degree-vtcs}{(1)}]
            For $H' \coloneqq H[V(H) \backslash S]$, we have
            \begin{equation*}
                \ind(H) \leq \frac{s^s}{s!e^s} \cdot \ind(H') + o_k(1) \leq \frac{s^s}{s!e^s} + o_k(1).
            \end{equation*}
        \item[\namedlabel{item:part2-lemma:ind-if-high-degree-vtcs}{(2)}]
            If $T \coloneqq \{v \in V(H) \backslash S \mid \neighbors[S]{v} \neq S\}$ has size at least $t$, then
            \begin{equation*}
                \ind(H) \leq \frac{s^s}{s!e^s} \cdot \frac{t^t}{t!e^t} + o_k(1).
            \end{equation*}
    \end{itemize}
\end{lemma}

\begin{lemma}\label{lemma:ind-if-almost-uniform-low-maximum-degree}
    Fix positive real numbers $\epsilon$ and $\beta$ and a positive integer $\tau$ and define $f \coloneqq \left\lfloor\sqrt{\frac{\beta}{\tau\epsilon}}\right\rfloor$. Consider a positive integer $k$ and a graph $H$ with $k$ vertices. Assume that $\Delta(H) \leq \epsilon k$ and that at least $\beta k$ vertices in $H$ have degree exactly $\tau$. Then, we have
    \begin{equation*}
        \ind(H) \leq 2 \left(\frac{1}{\beta} \cdot \frac{\tau^\tau}{\tau!e^\tau}\right)^f + o_k(1).
    \end{equation*}
\end{lemma}

The following lemma is a consequence of Martinsson et al.~\cite[Claim 3.3]{martinsson-mousset-noever-trujic:2020}. An alternative proof can be found in the author's bachelor's thesis~\cite[Section 6]{ueltzen:bs-thesis}.
\begin{lemma}\label{lemma:ind-if-not-almost-uniform}
    Fix positive real numbers $C$ and $\beta < 1$. Consider positive integers $k$ and $\ell$ with $\ell \leq Ck$ and a graph $H$ with $k$ vertices and $\ell$ edges. Let us assume that for all nonnegative integers $i$, we have $k_i \coloneqq |\{v \in V(H) \mid \degree{v} = i\}| \leq \beta k$.
    Then, it follows that
    \begin{equation*}
        \ind(H) \leq o_k(1).
    \end{equation*}
\end{lemma}

Martinsson et al.~[8, Claim 3.3] showed that for $\ell \leq o_k(k^{6/5})$, for a graph $G$ on $n$ vertices and a uniformly random $W \in \binom{V(G)}{k}$, the event $|E(G[W])| = \ell$ is essentially contained in the event that $(1-o_k(1))k$ vertices in $G[W]$ have the same degree in $G[W]$ (here, "$\E$ essentially contained in $\F$" means that $\IP[\F \backslash \E] \leq o_k(1)+o_{n \mid k}(1)$). Taking a graph $H$ as in Lemma 3.3, and setting $\ell=|E(H)|$, this implies that the probability of having $|E(G[W])|=|E(H)|$ and having no more than $\beta k$ vertices in $G[W]$ of the same degree is $o_k(1)+o_{n|k}(1)$. In particular, the event $G[W] \cong H$ has probability $o_k(1)+o_{n|k}(1)$, meaning that $\ind(H) = o_k(1)$ as stated in Lemma 3.3.

The following claim shows when the conditions for applying \cref{lemma:ind-if-high-degree-vtcs} can be met.
\begin{claim}\label{claim:conditions-for-lemma:ind-if-high-degree-vtcs}
    For any fixed positive real numbers $C$ and $\epsilon$, there exists a fixed $\delta = \delta(C, \epsilon) > 0$ satisfying the following property: Consider positive integers $k$ and $\ell$ with $\ell \leq Ck$ and a graph $H$ with $k$ vertices and $\ell$ edges, satsifying $\Delta(H) \geq \epsilon k$. Then, there exist positive real numbers $a, b$ with $a + \delta \leq b \leq \epsilon$ such that no vertex in $H$ has degree in the interval $(ak, bk)$ and $s \coloneqq |\{v \in V(H) \mid \degree{v} \geq bk\}|$ satisfies $1 \leq s \leq \frac{2C}{\delta}$.
\end{claim}
\begin{proof}
    We choose $\delta = \delta(C, \epsilon) > 0$ sufficiently small such that $\sum_{\substack{1 \leq i \leq \epsilon / \delta -1}} \delta i > 2C$.
    Therefore, by $\ell \leq Ck$,
    \begin{equation*}
        \sum_{\substack{1 \leq i \leq \epsilon / \delta -1}} \delta i k > 2Ck \geq 2\ell = \sum_{v \in V(H)} \degree{v}.
    \end{equation*}
    Hence, there exists a positive integer $i \leq \epsilon / \delta - 1$ such that no vertex in $H$ has degree in $(\delta i k, \delta (i+1) k)$. We define $a \coloneqq \delta i, b \coloneqq \delta (i+1)$ so that no vertex in $H$ has degree in $(ak, bk)$. Moreover, we have $b - a = \delta$ and $b = \delta(i+1) \leq \epsilon$. By $bk \leq \epsilon k \leq \Delta(H)$, it follows that $s = |\{v \in V(H) \mid \degree{v} \geq bk\}| \geq 1$. In addition,
    \begin{equation*}
        sbk \leq \sum_{v \in V(H)} \degree{v} = 2\ell \leq 2Ck,
    \end{equation*}
    where we used $\ell \leq Ck$. Since $b \geq b-a = \delta$, we conclude that
    \begin{equation*}
        s \leq \frac{2C}{b} \leq \frac{2C}{\delta}.\qedhere
    \end{equation*}
\end{proof}

\subsection{Proof of Theorem~\ref{theorem:main-result-for-denser-graphs}}\label{subsection:proof-of-theorem:main-result-for-denser-graphs}

Suppose that $\alpha > 0$ and $C$ are fixed. Let us fix $\epsilon = \epsilon(C) > 0$ such that
\begin{equation}\label{equation:epsilon-condition1}
    2 \left(\frac{2}{e}\right)^{\sqrt{1 / (8C\epsilon)} - 1} \leq \frac{2}{e^2}.
\end{equation}
Let $k$ and $\ell$ be positive integers with $\ell \leq Ck$ and let $H$ be a graph with $k$ vertices and $\ell$ edges satisfying $H \not\cong K_{1,k-1}$ and $m(H) \geq \alpha k$. We first consider the case $\Delta(H) \leq \epsilon k$. For any nonnegative integer $i$, denote the number of vertices with degree $i$ in $H$ by $k_i$ and define $\beta \coloneqq \max\{1-\frac{\alpha}{2}, \frac{1}{2}\} < 1$, so that $k_0 = k - m(H) \leq (1 - \alpha)k < \beta k$ by $m(H) \geq \alpha k$. If $k_i \leq \beta k$ for all nonnegative integers $i$, then \cref{lemma:ind-if-not-almost-uniform} with $\ell \leq Ck$ implies $\ind(H) \leq o_k(1)$.
Now, let us assume that for some nonnegative integer $\tau$, we have $k_\tau \geq \beta k$. Since $k_0 < \beta k$, it follows that $\tau \neq 0$ and thus, $\tau \geq 1$. By definition of $k_\tau$, we have $k_\tau \tau \leq \sum_{v \in V(H)} \degree{v} \leq 2\ell$. Together with $\beta \geq \frac{1}{2}$ and $k_\tau \geq \beta k$, it follows that
\begin{equation*}
    \frac{k \tau}{2} \leq \beta k \tau \leq k_\tau \tau \leq 2\ell \leq 2Ck,
\end{equation*}
where we used $\ell \leq Ck$.
Hence, $\tau \leq 4C$. \cref{lemma:poisson-point-probability-decreasing} with $\tau \geq 1$ implies $\frac{\tau^\tau}{\tau!e^\tau} \leq \frac{1^1}{1!e^1} = \frac{1}{e}$.
By $\beta \geq \frac{1}{2}$ and $\tau \leq 4C$, we have $f \coloneqq \left\lfloor\sqrt{\frac{\beta}{\tau\epsilon}}\right\rfloor \geq \sqrt{\frac{1}{8C\epsilon}} - 1$ and thus,
\begin{equation*}
    2 \left(\frac{1}{\beta} \cdot \frac{\tau^\tau}{\tau!e^\tau}\right)^f \leq 2 \left(\frac{2}{e}\right)^f \leq 2 \left(\frac{2}{e}\right)^{\sqrt{1 / (8C\epsilon)} - 1} \leq \frac{2}{e^2},
\end{equation*}
where we used $\frac{2}{e} \leq 1$ and \cref{equation:epsilon-condition1}. Since $k_\tau \geq \beta k$ and $\Delta(H) \leq \epsilon k$, it follows from \cref{lemma:ind-if-almost-uniform-low-maximum-degree} with $1 \leq \tau \leq 4C \leq O(1)$ that
\begin{equation*}
    \ind(H) \leq 2 \left(\frac{1}{\beta} \cdot \frac{\tau^\tau}{\tau!e^\tau}\right)^f + o_k(1) \leq \frac{2}{e^2} + o_k(1).
\end{equation*}

We now consider the case that $\Delta(H) \geq \epsilon k$. By \cref{claim:conditions-for-lemma:ind-if-high-degree-vtcs}, and since $\epsilon = \epsilon(C)$ is fixed, there exists a fixed $\delta = \delta(C) > 0$ such that there are real numbers $a, b \in (0, 1)$ with $b-a \geq \delta$, where no vertex in $H$ has degree in the interval $(ak, bk)$, and $S \coloneqq \{v \in V(H) \mid \degree{v} \geq bk\}$ has size $s \coloneqq |S|$ with $1 \leq s \leq \frac{2C}{\delta} \leq O(1)$. \cref{lemma:ind-if-high-degree-vtcs} therefore implies $\ind(H) \leq \frac{s^s}{s!e^s} + o_k(1)$. For $s \geq 2$, \cref{lemma:poisson-point-probability-decreasing} implies
\begin{equation*}
    \ind(H) \leq \frac{s^s}{s!e^s} + o_k(1) \leq \frac{2^2}{2!e^2} + o_k(1) = \frac{2}{e^2} + o_k(1).
\end{equation*}
Now, suppose that $s \leq 1$ and therefore, $s = 1$. As in the statement of \cref{lemma:ind-if-high-degree-vtcs}, define $H' \coloneqq H[V(H) \backslash S]$ and $T \coloneqq \{v \in V(H) \backslash S \mid \neighbors[S]{v} \neq S\}$. If $|T| \geq 1$, then by \cref{lemma:ind-if-high-degree-vtcs} applied to $s = 1$ and $t = 1$,
\begin{equation*}
    \ind(H) \leq \frac{s^s}{s!e^s} \cdot \frac{t^t}{t!e^t} + o_k(1) = \frac{1}{e^2} + o_k(1) \leq \frac{2}{e^2} + o_k(1).
\end{equation*}
Now, assume that $|T| = 0$. By $s = 1$ and $H \not\cong K_{1,k-1}$, we conclude that $E(H') \neq \emptyset$. Moreover, $H'$ has $k-s = k-1$ vertices and $|E(H')| \leq |E(H)| \leq Ck$ edges, which implies that $H'$ is not complete for sufficiently large $k$. We conclude by \cref{theorem:edge-statistics-conjecture} that $\ind(H') \leq \frac{1}{e} + o_k(1)$.
Hence, by \cref{lemma:ind-if-high-degree-vtcs} and $s = 1$, we obtain
\begin{equation*}
    \ind(H) \leq \frac{s^s}{s!e^s} \cdot \ind(H') + o_k(1) \leq \frac{1}{e^2} + o_k(1) \leq \frac{2}{e^2} + o_k(1).
\end{equation*}

\subsection{Proof of Theorem~\ref{theorem:weak-characterization-of-high-inducibility-graphs}}\label{subsection:proof-of-theorem:weak-characterization-of-high-inducibility-graphs}

The proof is based on the following observation (see also \cref{figure:observation-H-tame}):
\begin{fact}\label{fact:s-t-u-tame-h}
    For a graph $H$ and a set $S \subseteq V(H)$, define $T \coloneqq \{v \in V(H) \backslash S \mid \neighbors[S]{v} \neq S\}$, and $U \coloneqq \{v \in V(H) \backslash S \mid \neighbors[V(H) \backslash S]{v} \neq \emptyset\}$. Then, $V_0 \coloneqq S \cup T \cup U$ tames $H$.
\end{fact}
\begin{proof}
    By the choice of $U$, $V(H) \backslash (S \cup U)$ is a stable set and each vertex in $U$ is adjacent to no vertex in $V(H) \backslash (S \cup U)$. Therefore, $V(H) \backslash V_0 \subseteq V(H) \backslash (S \cup U)$ is also a stable set. Suppose that $v \in V_0$. If $v \in S$, $v$ is adjacent to all vertices in $V(H) \backslash V_0$ by the choice of $T$ and if $v \in U$, $v$ is adjacent to no vertex in $V(H) \backslash V_0 \subseteq V(H) \backslash (S \cup U)$. Finally, if $v \in V_0 \backslash (S \cup U)$, then $v$ is not connected to any vertex in $V(H) \backslash V_0 \subseteq V(H) \backslash (S \cup U)$ since $V(H) \backslash (S \cup U)$ is a stable set.
\end{proof}
We will apply this fact for the set $S$ of high-degree vertices and use the bound $|V_0| \leq |S| + |T| + |U|$.
\begin{figure}[ht]
    \centering
    \includegraphics{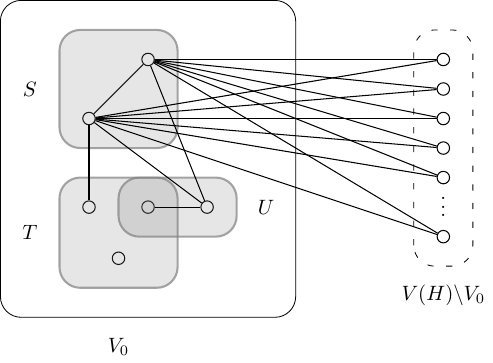}
    \caption{The graph $H$ is tamed by $V_0 = S \cup T \cup U$.}
    \label{figure:observation-H-tame}
\end{figure}

Let $\gamma > 0$ and $C$ be fixed and fix $\epsilon = \epsilon(C, \gamma) > 0$ such that
\begin{equation}\label{equation:epsilon-condition2}
    2 \left(\frac{2}{e}\right)^{\sqrt{1 / (16C\epsilon)} - 1} \leq \frac{\gamma}{2}.
\end{equation}
\begin{claim}\label{claim:weak-characterization-of-high-inducibility-low-max-degree-graphs}
    There exists a fixed positive integer $D_1 = D_1(\gamma)$ with the following property: Consider positive integers $k$ and $\ell$ with $\ell \leq 2Ck$ and a graph $H$ with $k$ vertices and $\ell$ edges. Suppose that $\Delta(H) \leq \epsilon k$ and $m(H) \geq D_1$. Then, we have
    \begin{equation*}
        \ind(H) \leq \frac{\gamma}{2} + o_k(1).
    \end{equation*}
\end{claim}
\begin{proof}
    Let $D_1 \coloneqq D(\frac{\gamma}{2})$ be the constant from \cref{theorem:fox-sauermann:2020}. For any nonnegative integer $i$, let $k_i$ be the number of vertices of degree $i$ in $H$ and define $\beta \coloneqq \frac{31}{32} \in [\frac{1}{2}, 1)$. If $k_i \leq \beta k$ for all integers $i$, then by \cref{lemma:ind-if-not-almost-uniform} with $\ell \leq 2Ck$, we have $\ind(H) \leq o_k(1)$.
    Now, let us assume that for some nonnegative integer $\tau$, we have $k_\tau \geq \beta k$. Therefore, $k_\tau \tau \leq \sum_{v \in V(H)} \degree{v} \leq 2\ell$ and hence, together with $\beta \geq \frac{1}{2}$ and $\ell \leq 2Ck$, it follows that
    \begin{equation*}
        \frac{k \tau}{2} \leq \beta k \tau \leq k_\tau \tau \leq 2\ell \leq 4Ck
    \end{equation*}
    and hence, $\tau \leq 8C$. If $\tau = 0$, we have $k_0 = k_\tau \geq \beta k$ or $m(H) = k - k_0 \leq (1-\beta)k = \frac{k}{32}$.
    By \cref{theorem:fox-sauermann:2020} with $m(H) \geq D_1$, we conclude $\ind(H) \leq \frac{\gamma}{2}$.
    Let us now assume that $\tau \neq 0$ and thus, $\tau \geq 1$.
    \cref{lemma:poisson-point-probability-decreasing} therefore implies $\frac{\tau^\tau}{\tau!e^\tau} \leq \frac{1}{e}$.
    By $\beta \geq \frac{1}{2}$ and $\tau \leq 8C$, we have $f \coloneqq \left\lfloor\sqrt{\frac{\beta}{\tau\epsilon}}\right\rfloor \geq \sqrt{\frac{1}{16C\epsilon}} - 1$ and thus,
    \begin{equation*}
        2 \left(\frac{1}{\beta} \cdot \frac{\tau^\tau}{\tau!e^\tau}\right)^f \leq 2 \left(\frac{2}{e}\right)^f \leq 2 \left(\frac{2}{e}\right)^{\sqrt{1 / (16C\epsilon)} - 1} \leq \frac{\gamma}{2},
    \end{equation*}
    where we used $\frac{2}{e} \leq 1$ and \cref{equation:epsilon-condition2}.
    Since $k_\tau \geq \beta k$ and $\Delta(H) \leq \epsilon k$, it follows from \cref{lemma:ind-if-almost-uniform-low-maximum-degree} with $1 \leq \tau \leq 8C \leq O(1)$ that
    \begin{equation*}
        \ind(H) \leq 2 \left(\frac{1}{\beta} \cdot \frac{\tau^\tau}{\tau!e^\tau}\right)^f + o_k(1) \leq \frac{\gamma}{2} + o_k(1).\qedhere
    \end{equation*}
\end{proof}

Define $D_1 = D_1(\gamma)$ as in the statement of \cref{claim:weak-characterization-of-high-inducibility-low-max-degree-graphs} and suppose that $D_2 = D_2(\gamma)$ is sufficiently large such that, by \cref{lemma:poisson-point-probability-decreasing}, $\frac{s^s}{s!e^s} \leq \frac{\gamma}{2}$ for all $s \geq D_2$.
Consider positive integers $k$ and $\ell$ with $\ell \leq Ck$ and a graph $H$ with $k$ vertices and $\ell$ edges such that $H$ is not $(2D_2 + D_1)$-tame. We now prove that
\begin{equation}\label{equation:weak-characterization-of-high-inducibility-graphs-goal}
    \ind(H) \leq \frac{\gamma}{2} + o_k(1).
\end{equation}

Since $H$ is not $(2D_2 + D_1)$-tame, we have $m(H) > 2D_2 + D_1 \geq D_1$. If $\Delta(H) \leq \epsilon k$, then \cref{equation:weak-characterization-of-high-inducibility-graphs-goal} follows by \cref{claim:weak-characterization-of-high-inducibility-low-max-degree-graphs} with $\ell \leq Ck \leq 2Ck$.
We now assume that $\Delta(H) \geq \epsilon k$. Let $\delta = \delta(C, \epsilon / 2) > 0$ be the constant in \cref{claim:conditions-for-lemma:ind-if-high-degree-vtcs}. Therefore, $\delta$ only depends on $C$ and $\gamma$ since $\epsilon$ only depended on $C$ and $\gamma$.
By \cref{claim:conditions-for-lemma:ind-if-high-degree-vtcs} with $\Delta(H) \geq \epsilon k \geq \frac{\epsilon}{2}k$ and $\ell \leq Ck$, we can find positive real numbers $a, b$ with $a < b \leq \frac{\epsilon}{2}$ and $b - a \geq \delta$ such that no vertex in $H$ has degree in the interval $(ak, bk)$ and $S \coloneqq \{V(H) \mid \degree{v} \geq bk\}$ has size $s \coloneqq |S|$ with $1 \leq s \leq \frac{2C}{\delta}$.
As in the statement of \cref{lemma:ind-if-high-degree-vtcs}, define $H' \coloneqq H[V(H) \backslash S]$, and $T \coloneqq \{v \in V(H) \backslash S \mid \neighbors[S]{v} \neq S\}$.
By \cref{fact:s-t-u-tame-h}, $S \cup T \cup V(M(H'))$ tames $H$.
Therefore, since $H$ is not $(2D_2+D_1)$-tame,
\begin{equation*}
    2D_2 + D_1 < |S \cup T \cup V(M(H'))| \leq |S| + |T| + m(H'),
\end{equation*}
and thus, at least one of $|S| > D_2$, $|T| > D_2$ or $m(H') > D_1$ holds. If $|S| > D_2 > 0$ or $|T| > D_2 > 0$, then by \cref{lemma:ind-if-high-degree-vtcs} and the definitions of $S$, $T$, and $D_2$, \cref{equation:weak-characterization-of-high-inducibility-graphs-goal} follows.
In the remaining case, we therefore have $m(H') > D_1$. Since $H'$ has $k - s \geq k - \frac{2C}{\delta}$ vertices, we conclude that for $k \geq \frac{4C}{\delta}$, we have $|V(H')| \geq \frac{k}{2}$. For $k \geq \frac{4C}{\delta}$ and any vertex $v \in V(H') = V(H) \backslash S$,
\begin{equation*}
    \degree[][H']{v} \leq \degree[][H]{v} \leq bk \leq \frac{\epsilon k}{2} \leq \epsilon |V(H')|,
\end{equation*}
which implies $\Delta(H') \leq \epsilon |V(H')|$, and
\begin{equation*}
    |E(H')| \leq \ell \leq Ck \leq 2C|V(H')|.
\end{equation*}
Since we have $\Delta(H') \leq \epsilon |V(H')|$, $|E(H')| \leq 2C|V(H')|$, and $m(H') > D_1$, \cref{claim:weak-characterization-of-high-inducibility-low-max-degree-graphs} implies that $\ind(H') \leq \frac{\gamma}{2} + o_k(1)$.
Hence, \cref{equation:weak-characterization-of-high-inducibility-graphs-goal} follows by \cref{lemma:ind-if-high-degree-vtcs} with $1 \leq s \leq \frac{2C}{\delta} \leq O(1)$.

Therefore, there exists a fixed $K = K(\gamma, C)$ such that for $k \geq K$, we have $\ind(H) < \gamma$. Since all graphs with $k \leq K$ vertices are $K$-tame, the desired conclusion follows with $D \coloneqq \max\{K, 2D_2 + D_1\}$.

\section{Proof of Lemma~\ref{lemma:ind-if-high-degree-vtcs}}\label{section:proof-of-lemma:ind-if-high-degree-vtcs}

Consider a graph $G$ on $n \geq k$ vertices, and let $W \in \binom{V(G)}{k}$ be chosen uniformly at random. We denote the vertices of low and high degrees in $G$ by
\begin{align*}
    V^+ & \coloneqq \left\{v \in V(G) \mmid \degree{v} \geq \frac{a+b}{2} \cdot n\right\}, \\
    V^- & \coloneqq \left\{v \in V(G) \mmid \degree{v} < \frac{a+b}{2} \cdot n\right\},
\end{align*}
respectively, and define
\begin{align*}
    W^+ & \coloneqq \{v \in W \mid \degree[W]{v} \geq bk\}, \\
    W^- & \coloneqq \{v \in W \mid \degree[W]{v} \leq ak\}.
\end{align*}
We will show that the event
\begin{equation*}
    \B \coloneqq \{W^- \cap V^+ \neq \emptyset\} \cup \{W^+ \cap V^- \neq \emptyset\}
\end{equation*}
holds with low probability.
If $\B^c$ holds, then $G[W] \cong H$ implies that $W^+ \dcup W^- = W$ and therefore, $W \cap V^+ = W^+$ has size exactly $s$. The probability for this last event can be bounded with \cref{lemma:poisson-bound-for-hypergeometric-distribution}. Conditioning on that event, we will upper-bound the probability of $\{G[W] \cong H\} \cap \B^c$. To prove the first part of the lemma, we bound the conditional probability that  $W \cap V^-$ induces a copy of $H'$ in $G$, and to prove the second part, we bound the probability that the induced bipartite graph on the sets $W \cap V^-$ and $W \cap V^+$ is a copy of the induced bipartite graph on the sets $S$ and $V(H) \backslash S$ in $H$.

\begin{claim}\label{claim:no-high-low-degree-swap-whp}
    We have $\IP[\B] \leq o_k(1)$.
\end{claim}
\begin{proof}
    For each $v \in V$, we define the event
    \begin{equation*}
        \B_v \coloneqq
        \begin{cases}
            \left\{v \in W^-\right\} & \text{for $v \in V^+$,} \\
            \left\{v \in W^+\right\} & \text{for $v \in V^-$,}
        \end{cases}
    \end{equation*}
    so that $\B = \bigcup_{w \in W} \B_v$.
    Let us bound $\IP[\B_v \mid v \in W]$ for an arbitrary vertex $v \in V(G)$. By symmetry, $W \backslash \{v\}$ conditioned on $v \in W$ is uniformly random from $\binom{V(G) \backslash \{v\}}{k-1}$. Hence, the distribution of $\degree[W]{v}$ conditioned on $v \in W$ is hypergeometric with sample size $k-1$ in a population of size $n-1$ containing $\degree{v}$ successes. Thus,
    \begin{equation*}
        E_v \coloneqq \IE[\degree[W]{v} \mid v \in W] = \degree{v} \cdot \frac{k-1}{n-1} = \degree{v} \cdot \frac{k}{n} + O(1).
    \end{equation*}
    Since $b - \frac{a+b}{2} = \frac{a+b}{2} - a = \frac{b-a}{2} \geq \frac{\delta}{2}$, we obtain
    \begin{align*}
        ak & \leq \frac{a+b}{2} \cdot k - \frac{\delta}{2} \cdot k \leq E_v - \left(\frac{\delta}{2}- o_k(1)\right)(k-1) \quad \text{for $v \in V^+$,}  \\
        bk & \geq \frac{a+b}{2} \cdot k + \frac{\delta}{2} \cdot k \geq E_v + \left(\frac{\delta}{2}- o_k(1)\right)(k-1)  \quad \text{for $v \in V^-$,}
    \end{align*}
    where we used the definitions of $V^+$ and $V^-$. Together with the definitions of $\B_v, W^+, W^-$, we obtain
    \begin{equation*}
        \B_v \subseteq
        \begin{cases}
            \{\degree[W]{v} \leq E_v -\left(\frac{\delta}{2}- o_k(1)\right) (k-1)\} & \text{for $v \in V^+$,} \\
            \{\degree[W]{v} \geq E_v +\left(\frac{\delta}{2}- o_k(1)\right) (k-1)\} & \text{for $v \in V^-$.}
        \end{cases}
    \end{equation*}
    Therefore, by tail bounds of the hypergeometric distribution (see, for example, \cite[Theorem 2.10]{janson-luczak-rucinski:2000}),
    \begin{equation*}
        \IP[\B_v \mid v \in W] \leq e^{-2(k-1)(\delta / 2 - o_k(1))^2} \leq o_k\left(\frac{1}{k}\right).
    \end{equation*}
    Since $\B = \bigcup_{w \in W} \B_v$, we conclude that
    \begin{equation*}
        \IP[\B] \leq \sum_{v \in V(G)} \IP[v \in W] \cdot \IP[\B_v \mid v \in W] \leq \sum_{v \in V(G)} \frac{k}{n} \cdot o_k\left(\frac{1}{k}\right) \leq o_k(1).\qedhere
    \end{equation*}
\end{proof}
Let us define $\hat t \coloneqq \left|T\right|$, and $\overline{T} \coloneqq \{v \in W \cap V^- \mid \neighbors[W \cap V^+]{v} \neq W \cap V^+\}$, and events
\begin{align*}
    \A_1       & \coloneqq \{|W \cap V^+| = s\},                                                                                        \\
    \A_2       & \coloneqq \{G[W \cap V^-] \cong H'\},                                                                                  \\
    \A_3       & \coloneqq \{|\overline{T}| = \hat t\},                                                                                 \\
    \A_4(\tau) & \coloneqq \bigcup_{w \in W \cap V^+}\{\degree[W \cap V^-]{w} = \tau\} \quad \text{for any nonnegative integer $\tau$,} \\
    \A_4       & \coloneqq \bigcap_{v \in S} \A_4(\degree[V(H) \backslash S]{v}).
\end{align*}

\begin{claim}\label{claim:event-contained-in-intersection}
    We have $\{G[W] \cong H\} \cap \B^c \subseteq \A_1 \cap \A_2 \cap \A_3 \cap \A_4$.
\end{claim}
\begin{proof}
    Assume that $G[W] \cong H$ and $\B^c$ hold. By $G[W] \cong H$, the labeled graphs $(G[W], W^+, W^-)$ and $(H, S, V(H) \backslash S)$ are isomorphic, i.e., there exists a graph isomorphism $\phi$ from $G[W]$ to $H$ such that the images of $W^+$ and $W^-$ under $\phi$ are $S$ and $V(H) \backslash S$, respectively. Hence, no vertex in $G[W]$ has degree in $(ak, bk)$, which yields $W^+ = W \cap V^+$ and $W^- = W \cap V^-$ by $\B^c$. We conclude that $|W \cap V^+| = |W^+| = |S| = s$ and $G[W \cap V^-] = G[W^-] \cong H[V(H) \backslash S] = H'$ and thus, $\A_1, \A_2$ hold. Moreover,
    \begin{equation*}
        \overline{T} = \{v \in W^- \mid \neighbors[W^+]{v} \neq W^+\} \cong \{v \in V(H) \backslash S \mid \neighbors[S]{v} \neq S\} = T.
    \end{equation*}
    Therefore, we have $|\overline{T}| = |T| = \hat t$ and $\A_3$ follows. For any $v \in S$, there exists a $w \in W^+ = W \cap V^+$ with $\degree[W \cap V^-]{w} = \degree[W^-]{w} = \degree[V(H) \backslash S]{v}$. Thus, $\A_4(\degree[V(H) \backslash S]{v})$ follows.
\end{proof}
By \cref{claim:no-high-low-degree-swap-whp,claim:event-contained-in-intersection}, we obtain
\begin{equation}\label{equation:upper-bound-induced-density-with-high-degree-vtcs}
    \indd{H}{G} \leq \IP[\A_1 \cap \A_2 \cap \A_3 \cap \A_4] + \IP[B] \leq \IP[\A_1 \cap \A_2 \cap \A_3 \cap \A_4] + o_k(1).
\end{equation}
We will use that $s \leq O(1)$, $t \leq O(1)$, as $s$ and $t$ are fixed. Define $\sigma = \frac{|V^+|}{n} \in [0, 1]$, so that $|W \cap V^+|$ has a hypergeometric distribution with sample size $k$ in a population of size $n$ containing $\sigma n$ successes.
\begin{claim}\label{claim:induced-density-for-sigma-large}
    For $\sigma \geq \frac{1}{2}$, we have $\indd{H}{G} \leq o_k(1) + o_{n \mid k}(1)$.
\end{claim}
\begin{proof}
    By \cref{lemma:poisson-bound-for-hypergeometric-distribution}, we have
    \begin{equation*}
        \IP[\A_1] = \binom{k}{s} \sigma^{s} (1-\sigma)^{k-s} + o_{n \mid k}(1) \leq k^s \left(\frac{1}{2}\right)^{k-s}  + o_{n \mid k}(1) \leq o_k(1) + o_{n \mid k}(1),
    \end{equation*}
    where we used $s \leq O(1)$ and $\sigma \geq \frac{1}{2}$.
    By \cref{equation:upper-bound-induced-density-with-high-degree-vtcs}, $\indd{H}{G} \leq \IP[\A_1] + o_k(1) \leq o_k(1) + o_{n \mid k}(1)$.
\end{proof}
It follows from \cref{lemma:poisson-bound-for-hypergeometric-distribution} with $1 \leq s \leq O(1)$ that
\begin{equation}\label{equation:upper-bound-point-probability-for-high-degree-vtcs}
    \IP[\A_1] \leq \frac{s^{s}}{s! e^{s}} + o_k(1) + o_{n \mid k}(1).
\end{equation}

\begin{claim}\label{claim:bound-of-conditional-a2-probability}
    For $\sigma \leq \frac{1}{2}$, we have $\IP[\A_2 \mid \A_1] \leq \ind(H') + o_{n \mid k}(1)$.
\end{claim}
\begin{proof}
    Recall that $\A_1 = \{|W \cap V^+| = s\}$ only depends on $W \cap V^+$. Therefore, conditioned on $\A_1$, the distribution of $W \cap V^-$ is uniform on $\binom{V^-}{k-s}$. Hence,
    \begin{equation*}
        \IP[\A_2 \mid \A_1] = \indd{H'}{G[V^-]} \leq \ind(H') + o_{n \mid k}(1),
    \end{equation*}
    where we used that $|V(H')| = k - s = k - O(1)$, and $|V(G[V^-])| = |V^-| = (1-\sigma)n \geq \frac{n}{2}$.
\end{proof}
\cref{equation:upper-bound-point-probability-for-high-degree-vtcs,claim:bound-of-conditional-a2-probability} show that for $\sigma \leq \frac{1}{2}$, we have
\begin{equation*}
    \IP[\A_1 \cap \A_2] = \IP[\A_1] \cdot \IP[\A_2 \mid \A_1] \leq \frac{s^s}{s!e^s} \cdot \ind(H') + o_k(1) + o_{n \mid k}(1)
\end{equation*}
and thus, by \cref{equation:upper-bound-induced-density-with-high-degree-vtcs},
\begin{equation*}
    \indd{H}{G} \leq \IP[\A_1 \cap \A_2] + o_k(1) \leq \frac{s^s}{s!e^s} \cdot \ind(H') + o_k(1) + o_{n \mid k}(1).
\end{equation*}
Since \cref{claim:induced-density-for-sigma-large} implies the same bound for $\sigma \geq \frac{1}{2}$, the first part of the lemma now follows by taking $n \to \infty$, as $G$ was an arbitrary $n$-vertex graph.

We now prove the second part of the lemma. Hence, we assume that $\hat t = |T| \geq t$. We will upper-bound the probability that the induced bipartite graph with parts $W \cap V^-$ and $W \cap V^+$ is a copy of the induced bipartite graph of $H$ with parts $S$ and $V(H) \backslash S$. To that end, we will consider the distribution of $W \cap V^-$ conditioned on the event $\{W \cap V^+ = W_0\}$ for an arbitrary $W_0 \in \binom{V^+}{s}$. By symmetry and \cref{equation:upper-bound-point-probability-for-high-degree-vtcs}, we have
\begin{equation}\label{equation:upper-bound-point-probability-for-high-degree-vtx-set}
    \IP[W \cap V^+ = W_0] = \binom{\sigma n}{s}^{-1} \IP[\A_1] \leq \binom{\sigma n}{s}^{-1}\left(\frac{s^s}{s!e^s} + o_k(1) + o_{n \mid k}(1)\right).
\end{equation}
\begin{claim}\label{claim:bound-of-conditional-a3-a4-probability}
    For $\sigma \leq \frac{1}{2}$, we have $\IP[\A_3 \cap \A_4 \mid W \cap V^+ = W_0] \leq \frac{t^{t}}{t!e^t} + o_k(1) + o_{n \mid k}(1)$.
\end{claim}
\begin{proof}
    By \cref{lemma:poisson-point-probability-decreasing}, there exists a fixed positive integer $P = P(s, t)$ such that
    \begin{equation}\label{equation:defining-inequality-for-P}
        \frac{P^P}{P!e^P} \leq \frac{1}{s} \cdot \frac{t^t}{t! e^t}.
    \end{equation}
    Consider the case $\hat t \leq sP$. Conditioned on $W \cap V^+ = W_0$, the distribution of $W \cap V^-$ is uniform on $\binom{V^-}{k-s}$ and thus, $|\overline{T}|$ has a hypergeometric distribution with sample size $k - s$ in a population of size $|V^-|$ containing $|\{v \in V^- \mid \neighbors[W_0]{v} \neq W_0\}|$ successes.
    By \cref{lemma:poisson-bound-for-hypergeometric-distribution,lemma:poisson-point-probability-decreasing} with $1 \leq t \leq \hat t \leq sP \leq O(1)$,
    \begin{equation*}
        \IP[\A_3 \mid W \cap V^+ = W_0] \leq \frac{\hat t^{\hat t}}{\hat t! e^{\hat t}} + o_k(1) + o_{n \mid k}(1) \leq \frac{t^{t}}{t!e^t} + o_k(1) + o_{n \mid k}(1),
    \end{equation*}
    where we used that $k - s = k - O(1)$ and $|V^-| = (1-\sigma)n \geq \frac{n}{2}$.
    By $\A_3 \cap \A_4 \subseteq \A_3$, the claim follows.

    Let us consider the remaining case $\hat t \geq sP$.
    Since $T = \bigcup_{v \in S} ((V(H) \backslash S) \backslash \neighbors{v})$, we can find a $v \in S$ such that $\tau \coloneqq \degree[V(H) \backslash S]{v}$ satisfies
    \begin{equation*}
        k-s-\tau = |(V(H) \backslash S) \backslash \neighbors{v}| \geq \frac{|T|}{|S|} = \frac{\hat t}{s} \geq P.
    \end{equation*}
    We may assume that $k \geq \frac{P + s}{\delta}$. Therefore, since $v \in S$, we also have $\tau \geq \degree{v} - s \geq bk - s \geq \delta k - s \geq P$. Now, fix a vertex $w_0 \in W_0$.
    Conditioned on $W \cap V^+ = W_0$, the distribution of $W \cap V^-$ is uniform on $\binom{V^-}{k-s}$ and therefore, the conditional distribution of $\degree[W \cap V^-]{w_0}$ is hypergeometric with sample size $k - s$ in a population of size $|V^-|$ containing $\degree[V^-]{w_0}$ successes. Since $P$ is fixed, \cref{lemma:poisson-bound-for-hypergeometric-distribution} with $\min\{\tau, k-s-\tau\} \geq P$ implies
    \begin{equation*}
        \IP[\degree[W \cap V^-]{w_0} = \tau \mid W \cap V^+ = W_0] \leq \frac{P^P}{P!e^P} + o_{n \mid k}(1),
    \end{equation*}
    where we used $k - s = k - O(1)$ and $|V^-| \geq \Omega_n(n)$.
    Summing this inequality over all $w_0 \in W_0$ yields
    \begin{align*}
        \IP[\A_4 \mid W \cap V^+ = W_0]
         & \leq \IP[\A_4(\tau) \mid W \cap V^+ = W_0]                                         \\
         & \leq \sum_{w_0 \in W_0} \IP[\degree[W \cap V^-]{w_0} = \tau \mid W \cap V^+ = W_0] \\
         & \leq |W_0|\left(\frac{P^P}{P!e^P} + o_{n \mid k}(1)\right)                         \\
         & \leq \frac{t^t}{t!e^t} + o_{n \mid k}(1),
    \end{align*}
    where we used $|W_0| = s$ and \cref{equation:defining-inequality-for-P}. The desired claim follows by $\A_3 \cap \A_4 \subseteq \A_4$.
\end{proof}
For $\sigma \leq \frac{1}{2}$, \cref{claim:bound-of-conditional-a3-a4-probability,equation:upper-bound-point-probability-for-high-degree-vtx-set} imply
\begin{align*}
    \IP[\A_1 \cap \A_3 \cap \A_4]
     & = \sum_{W_0 \in \binom{V^+}{s}} \IP[\A_3 \cap \A_4 \mid W \cap V^+ = W_0] \cdot \IP[W \cap V^+ = W_0]                                                                    \\
     & \leq \binom{\sigma n}{s} \left(\frac{t^{t}}{t!e^t} + o_k(1) + o_{n \mid k}(1)\right) \binom{\sigma n}{s}^{-1} \left(\frac{s^s}{s!e^s}  + o_k(1) + o_{n \mid k}(1)\right) \\
     & \leq \frac{s^s}{s!e^s} \cdot \frac{t^t}{t!e^t} + o_k(1) + o_{n \mid k}(1)
\end{align*}
and thus, by \cref{equation:upper-bound-induced-density-with-high-degree-vtcs},
\begin{equation*}
    \indd{H}{G} \leq \IP[\A_1 \cap \A_3 \cap \A_4] + o_k(1) \leq \frac{s^s}{s!e^s} \cdot \frac{t^t}{t!e^t} + o_k(1) + o_{n \mid k}(1).
\end{equation*}
Since \cref{claim:induced-density-for-sigma-large} implies the same bound for $\sigma \geq \frac{1}{2}$, the second part of the lemma now follows by taking $n$ to infinity, as $G$ was an arbitrary $n$-vertex graph.

\section{Proof of Lemma~\ref{lemma:ind-if-almost-uniform-low-maximum-degree}}\label{section:proof-of-lemma:ind-if-almost-uniform-low-maximum-degree}

By definition, $f$ is a nonnegative integer and the claim is trivial for $f = 0$. From now on, we assume that $f$ is a positive integer. Let $n$ be a positive integer with $n \geq k$ and let $G$ be a graph on $n$ vertices. We consider a uniformly random set $W \in \binom{V(G)}{k}$ so that $\IP[G[W] \cong H] = \indd{H}{G}$.
Let
\begin{equation*}
    U \coloneqq \{v \in W \mid \degree[W]{v} = \tau\}.
\end{equation*}
For a vertex $v \in V(G)$, we can upper-bound the probability
\begin{equation*}
    \IP[v \in U \mid v \in W] = \IP[|(W \backslash \{v\}) \cap N(v)| = \tau \mid v \in W]
\end{equation*}
by \cref{lemma:poisson-bound-for-hypergeometric-distribution}. This yields an upper bound on $\IE[|U|]$.
We will improve this upper bound by considering, for every $A \in \binom{V(G)}{f}$, the events $(\{v \in U\})_{v \in A}$. If those events, conditioned on $A \subseteq W$, were independent, then, for large $f$, the probability of $\{A \subseteq U\} = \bigcap_{v \in A} \{v \in U\}$ would be very small, which would imply a much better upper bound on $\IE[|U|]$. Evidently, this is not the case. However, the events $(\{v \in U\})_{v \in A}$ are (informally) close to independent conditioned on the event that $W$ contains all vertices in $A$ but no vertex having at least two neighbors in $A$. This motivates upper-bounding the expected size of
\begin{equation*}
    F \coloneqq \left\{A \in \binom{U}{f} \mmid \text{for any}~\{v, w\} \in \binom{A}{2}: \{v, w\} \not \in E(G), \neighbors[W]{v} \cap \neighbors[W]{w} = \emptyset \right\}.
\end{equation*}
In the second part of the proof, we lower-bound $|F|$ on the event $\{G[W] \cong H\}$.
In this step, we crucially use $\Delta(H) \leq \epsilon k$.
Combining the bounds on $\IE[|F|]$ yields the required upper bound on $\indd{H}{G}$.
\begin{claim}\label{claim:bound-for-p}
    For $A \in \binom{V(G)}{f}$, we have
    \begin{equation*}
        \IP[A \in F \mid A \subseteq W] \leq \left(\frac{\tau^\tau}{\tau!e^\tau}\right)^f + o_k(1) + o_{n \mid k}(1).
    \end{equation*}
\end{claim}
\begin{proof}
    If $\{v, w\} \in E(G[A])$, we have $A \not\in F$ for any choice of $W$ and thus, $\IP[A \in F \mid A \subseteq W] = 0$. Let us now assume that $E(G[A]) = \emptyset$. Therefore, $N_1 \coloneqq \{v \in V(G) \mid \degree[A]{v} = 1\}$
    satisfies $N_1 \cap A = \emptyset$. Hence, by definition of $F$ and $N_1$,
    \begin{equation*}
        \{A \in F\} \subseteq \bigcap_{v \in A} \{\degree[N_1 \cap W]{v} = \tau\} = \bigcap_{v \in A} \{|\neighbors[N_1]{v} \cap (W \backslash A)| = \tau\},
    \end{equation*}
    which implies
    \begin{equation*}
        \IP[A \in F \mid A \subseteq W] \leq \IP\left[\bigcap_{v \in A} \{|\neighbors[N_1]{v} \cap (W \backslash A)| = \tau\} \mmid A \subseteq W\right].
    \end{equation*}
    By symmetry, the distribution of $W \backslash A$ conditioned on $A \subseteq W$ is uniform on $\binom{V(G) \backslash A}{k - f}$.
    Since $(\neighbors[N_1]{v})_{v \in A}$ are disjoint subsets of $V(G) \backslash A$, \cref{lemma:poisson-bound-for-multi-hypergeometric-distribution} yields
    \begin{equation*}
        \IP\left[\bigcap_{v \in A} \{|\neighbors[N_1]{v} \cap (W \backslash A)| = \tau\} \mmid A \subseteq W\right] \leq \left(\frac{\tau^\tau}{\tau!e^\tau}\right)^f + o_k(1) + o_{n \mid k}(1),
    \end{equation*}
    where we used that $\tau$ and $f$ are fixed positive integers and thus, $k - f = k - O(1)$ and $n-f = n - O(1)$.
    The desired conclusion follows from the previous two equations.
\end{proof}
Since $\{A \in F\} \subseteq \{A \subseteq W\}$ and $\IP[A \subseteq W] = \binom{k}{f} \binom{n}{f}^{-1}$, it follows from \cref{claim:bound-for-p} that
\begin{equation}\label{equation:ex-of-f-upper}
    \IE[|F|] = \sum_{A \in \binom{V(G)}{f}} \IP[A \in F \mid A \subseteq W] \cdot \IP[A \subseteq W] \leq \left(\left(\frac{\tau^\tau}{\tau!e^\tau}\right)^f + o_k(1) + o_{n \mid k}(1)\right) \binom{k}{f}.
\end{equation}
\begin{claim}\label{claim:f-large}
    If $G[W] \cong H$, then $|F| \geq \left(\frac{1}{2} - o_k(1)\right) \beta^f \binom{k}{f}$.
\end{claim}
\begin{proof}
    Assume that $G[W] \cong H$. Therefore, our assumptions on $H$ imply
    \begin{align}
        \Delta(G[W])
         & = \Delta(H) \leq \epsilon k, \label{equation:g-w-low-max-degree} \\
        |U|
         & \geq \beta k. \label{equation:u-large}
    \end{align}
    Define
    \begin{align*}
        Q_1 & \coloneqq \left\{A \in \binom{U}{f} \mmid E(G[A]) \neq \emptyset \right\},                               \\
        Q_2 & \coloneqq \left\{A \in \binom{U}{f} \mmid \text{there exists a}~v \in W : \degree[A]{v} \geq 2 \right\},
    \end{align*}
    so that
    \begin{equation}\label{equation:f-repr}
        F = \binom{U}{f} \backslash (Q_1 \cup Q_2).
    \end{equation}
    Double-counting the number of pairs $(A, e)$ with $A \in Q_1, e \in E(G[A])$ yields
    \begin{equation}\label{equation:q1-bound}
        \left|Q_1\right| \leq \sum_{e \in E(G[U])} \binom{|U| - 2}{f - 2} \leq \left(\frac{\tau|U|}{2}\right) \binom{|U|-2}{f-2},
    \end{equation}
    where we used that $|E(G[U])| = \frac{1}{2}\sum_{v \in U} \degree[U]{v} \leq \frac{1}{2}\sum_{v \in U} \degree[W]{v} = \frac{\tau|U|}{2}$. Double-counting the number of tuples $(A, v, g)$ with $A \in Q_2$ and $v \in W$ and $g \in \binom{\neighbors[A]{v}}{2}$ yields
    \begin{equation}\label{equation:q2-bound}
        \left|Q_2\right|
        \leq \sum_{v \in W} \binom{\degree[U]{v}}{2} \binom{|U|-2}{f-2}
        \leq \sum_{v \in W} \frac{\epsilon k \cdot \degree[U]{v}}{2} \cdot \binom{|U|-2}{f-2}
        = \frac{\tau\epsilon k|U|}{2} \cdot \binom{|U|-2}{f-2},
    \end{equation}
    where we used that the degree of a vertex $v \in W$ into $U$ satisfies $\degree[U]{v} \leq \degree[W]{v} \leq \epsilon k$ by \cref{equation:g-w-low-max-degree} and that $\sum_{v \in W} \degree[U]{v} = \sum_{u \in U} \degree[W]{u} = \tau|U|$.
    From \cref{equation:f-repr,equation:q1-bound,equation:q2-bound}, we obtain
    \begin{equation*}
        |F| \geq \binom{|U|}{f} - \frac{\tau |U|}{2} \cdot \binom{|U|-2}{f-2} - \frac{\tau\epsilon k |U|}{2} \cdot \binom{|U|-2}{f-2}.
    \end{equation*}
    \cref{equation:u-large} implies that $|U| - 1 \geq \beta k - 1 \geq \Omega_k(k)$ and thus, $\frac{k|U|}{2} \leq (\beta^{-1} + o_k(1))\binom{|U|}{2}$, where we used that $\beta$ is fixed.
    By $\binom{|U|}{2}\binom{|U|-2}{f-2} = \binom{|U|}{f}\binom{f}{2}$ combined with $\tau, \epsilon, \beta$ fixed and the previous equation, we obtain
    \begin{equation*}
        |F| \geq \left(1 - \tau\epsilon\beta^{-1} \binom{f}{2} - o_k(1)\right) \binom{|U|}{f}.
    \end{equation*}
    The claim follows since $\binom{|U|}{f} \geq (\beta^f - o_k(1)) \binom{k}{f}$ by \cref{equation:u-large} and since $\tau\epsilon\beta^{-1} \binom{f}{2} \leq \frac{1}{2}$ by $f = \left\lfloor\sqrt{\frac{\beta}{\tau\epsilon}}\right\rfloor$.
\end{proof}
Since $|F| \geq 0$, \cref{claim:f-large} and $\IP[G[W] \cong H] = \indd{H}{G}$ imply
\begin{equation}\label{equation:ex-of-f-lower}
    \IE[|F|] \geq \left(\frac{1}{2} - o_k(1)\right) \beta^f \binom{k}{f} \indd{H}{G}.
\end{equation}
By \cref{equation:ex-of-f-upper,equation:ex-of-f-lower}, we have
\begin{equation*}
    \left(\frac{1}{2} - o_k(1)\right)\beta^f \binom{k}{f} \indd{H}{G} \leq \IE[|F|] \leq \left(\left(\frac{\tau^\tau}{\tau!e^\tau}\right)^f + o_k(1) + o_{n \mid k}(1)\right) \binom{k}{f}.
\end{equation*}
Therefore, since $\beta$ and $f$ are fixed,
\begin{equation*}
    \indd{H}{G} \leq 2\left(\frac{1}{\beta} \cdot \frac{\tau^\tau}{\tau!e^\tau}\right)^f + o_k(1) + o_{n \mid k}(1).
\end{equation*}
The desired conclusion follows by taking $n$ to infinity, as $G$ was an arbitrary $n$-vertex graph.

\section{Proof of Theorem~\ref{theorem:main-result-for-sparser-graphs}}\label{section:proof-of-theorem:main-result-for-sparser-graphs}

We start with defining detectable vertices, which are crucial in \cref{lemma:ind-if-sparse} that we prove in \cref{section:proof-of-lemma:ind-if-sparse}. In the rest of this section, we will prove general properties of detectable vertices. In particular, \cref{claim:characterization-of-detectable-vertices} provides an equivalent explicit criterion for when a vertex is detectable.
\begin{definition}
    For a graph $H$ and a non-isolated vertex $v \in V(H)$, we say that $v$ is \textit{obscure in $H$} if there exist non-isolated, distinct vertices $v_1, v_2$ with $\degree[][H - v_1]{v_2} > 0$,
    and an induced subgraph $H'$ of $H - v_1 - v_2$ with $M(H - v) \cong H'$.
    We say that a non-isolated vertex $v  \in V(H)$ is \textit{detectable in $H$} if $v$ is not obscure in $H$.
\end{definition}

\begin{definition}
    The number of vertices with degree $1$ in $H$ is denoted by $m_1(H)$ and the number of vertices with degree at least $2$ in $H$ is denoted by $m_{\geq 2}(H)$.
    We say that a non-isolated vertex $v \in V(H)$ is \textit{happy in $H$} if all its neighbors have degree at least $2$.
\end{definition}
Any vertex of degree $1$ is adjacent to exactly one vertex. Furthermore, any non-isolated, non-happy vertex is adjacent to some vertex of degree $1$. This yields the following bound:
\begin{fact}\label{fact:num-happy-vtcs-bound}
    The number of happy vertices in a graph $H$ is at least $m(H) - m_1(H) = m_{\geq 2}(H)$.
\end{fact}

\begin{claim}\label{claim:characterization-of-detectable-vertices}
    A vertex $v \in V(H)$ is detectable in $H$ if and only if it satisfies at least one of the following:
    \begin{itemize}
        \item[\namedlabel{item:condition1-happy}{(1)}] $v$ is happy in $H$,
        \item[\namedlabel{item:condition2-degree1}{(2)}] $v$ has degree $1$.
    \end{itemize}
\end{claim}
\begin{proof}
    Assume that $v \in V(H)$ satisfies neither \ref{item:condition1-happy} nor \ref{item:condition2-degree1}. If $v$ is isolated, then it is not detectable by definition. Now, we consider the case that $v$ is non-isolated. Since \ref{item:condition2-degree1} does not hold, it follows that $v$ has degree at least $2$. Therefore, since \ref{item:condition1-happy} does not hold, $v$ has a neighbor $w$ of degree $1$. Thus, $v$ is the unique neighbor of $w$, which implies $M(H - v) \cong M(H - w - v)$.
    Since $v$ has degree at least $2$, we have $\degree[][H-w]{v} > 0$. Therefore, $v$ is obscure in $H$. Consequently, $v$ is not detectable in $H$.

    We now assume that $v$ is not detectable in $H$. If $v$ is isolated, \ref{item:condition1-happy} and \ref{item:condition2-degree1} trivially do not hold. Now, suppose that $v$ is non-isolated. Therefore, since $v$ is not detectable, $v$ is obscure in $H$. Thus, we can find distinct non-isolated vertices $v_1, v_2 \in V(H)$ such that $\degree[][H - v_1]{v_2} > 0$ and an induced subgraph $H'$ of $M(H - v_1 - v_2)$ with $M(H - v) \cong H'$.
    It follows that
    \begin{alignat*}{6}
        m(H - v)   & =\, & m(H')   & \leq\, & m(H - v_1 - v_2)   & \leq\, & m(H - v_1) - 1   & \leq\, & m(H) - 2,   \\
        |E(H - v)| & =\, & |E(H')| & \leq\, & |E(H - v_1 - v_2)| & \leq\, & |E(H - v_1)| - 1 & \leq\, & |E(H)| - 2,
    \end{alignat*}
    which imply that \ref{item:condition1-happy} and \ref{item:condition2-degree1} do not hold, respectively.
\end{proof}

Consider any graph $H$ with $\ell \geq 2$ edges and let $v \in V(H)$ be detectable in $H$. If $v$ has degree $1$, then $H - v$ has $\ell - 1 \geq 1$ edges. If $v$ does not have degree $1$, it follows by \cref{claim:characterization-of-detectable-vertices} and since $v$ is detectable in $H$ that $v$ is happy in $H$. Thus, $v$ has some neighbor $w$ and $w$ has degree $\degree{w} \geq 2$. Hence, $H - v$ has at least $\degree{w} - 1 \geq 1$ edges. So, we obtain the following:
\begin{fact}\label{fact:graph-minus-detectable-vertex}
    If a graph $H$ has $\ell \geq 2$ edges and $v$ is dectable in $H$, then $H - v$ has at least one edge.
\end{fact}

\begin{definition}
    Consider a positive integer $k$ and a graph $H$ on $k$ vertices. We say that a labeling $v_1, \dots, v_k$ of $V(H)$ is \textit{bright} if there are at least two indices $i$ with $1 \leq i \leq k$ and $\degree[\{v_1, \dots, v_i\}]{v_i} > 0$ and the two largest such indices $i_1 < i_2$ have the property that $v_{i_1}, v_{i_2}$ are detectable in $H$. We denote the probability that a uniformly random labeling $v_1, \dots, v_k$ of $V(H)$ is bright by $\brightness{H}$.
\end{definition}

The proof of the following lemma is deferred to \cref{section:proof-of-lemma:ind-if-sparse}.

\begin{lemma}\label{lemma:ind-if-sparse}
    Let $\alpha > 0$ be fixed. Consider positive integers $k$ and $\ell$ with $\ell \geq 2$ and a graph $H$ with $k$ vertices and $\ell$ edges such that $m(H) \leq \alpha k$. Then, we have
    \begin{equation*}
        \ind(H) \leq \frac{2 + 3e^2 \alpha}{2 + (e-2)\brightness{H}} \cdot \frac{1}{e} + 2\alpha.
    \end{equation*}
\end{lemma}

Let us fix $\alpha > 0$ such that
\begin{equation*}
    c \coloneqq \frac{2 + 3e^2 \alpha}{2 + (e-2)\frac{1}{12}} \cdot \frac{1}{e} + 2\alpha < \frac{1}{e}.
\end{equation*}
By $\alpha > 0$ and $\frac{1}{12} \leq 1$, we have $c \geq \frac{2}{2 + (e-2)} \cdot \frac{1}{e} = \frac{2}{e^2}$.
Let $k$ and $\ell$ be positive integers with $\ell \geq 2$ and let $H$ be a graph with $k$ vertices and $\ell$ edges such that $m(H) \leq \alpha k$.
\begin{remark}
    By \cref{fact:graph-minus-detectable-vertex}, if all non-isolated vertices in $H$ are detectable, then $\nu(H) = 1$. Under the additional assumption $m(H) \leq o_k(1)$, \cref{lemma:ind-if-sparse} implies the bound $\ind(H) \leq \frac{2}{e^2} + o_k(1)$ of \cref{conjecture:upper-bound-2-over-e-squared}.
\end{remark}

\begin{claim}\label{claim:beta-bound-by-m2}
    We have $\brightness{H} \geq \binom{m_{\geq 2}(H)}{2} \binom{m(H)}{2}^{-1}$.
\end{claim}
\begin{proof}
    Since $H$ has $\ell \geq 2 > 0$ edges, we have $m(H) \geq 2$. Consider a uniformly random labeling $v_1, \dots, v_k$ of $V(H)$ and define the event
    \begin{equation*}
        \A \coloneqq \{\text{the last two non-isolated terms among $v_1, \dots, v_k$ are happy}\}.
    \end{equation*}
    Let us assume that $\A$ holds. Hence, the last two non-isolated terms $v_{i_1}, v_{i_2}$, $i_1 < i_2$, among $v_1, \dots, v_k$ are happy. Moreover, it is not possible that $v_{i_2}$ is the unique neighbor of $v_{i_1}$ since $v_{i_2}$ is happy. Therefore, $v_{i_1}, v_{i_2}$ are both adjacent to a vertex that appears before them in the sequence $v_1, \dots, v_k$ and hence, they are the last two such vertices. Since $v_{i_1}, v_{i_2}$ are happy and thus, by \cref{claim:characterization-of-detectable-vertices}, detectable, we conclude that the sequence $v_1, \dots, v_k$ is bright.
    We obtain $\A \subseteq \{v_1, \dots, v_k~\text{is bright}\}$. Hence, the probability that $v_1, \dots, v_k$ is bright, which is given by $\nu(H)$, is lower bounded by $\IP[A]$. Therefore, \cref{fact:num-happy-vtcs-bound} implies
    \begin{equation*}
        \brightness{H} \geq \IP[\A]
        = \binom{|\{v \in V(H)~\text{happy}\}|}{2} \binom{m(H)}{2}^{-1} \geq \binom{m_{\geq 2}(H)}{2} \binom{m(H)}{2}^{-1}.\qedhere
    \end{equation*}
\end{proof}

\begin{claim}\label{claim:beta-bound-by-m1}
    We have
    \begin{equation*}
        \brightness{H} \geq \left(\frac{m_1(H)(m_1(H)-2)}{2}\right) \binom{m(H)}{2}^{-1}.
    \end{equation*}
    Furthermore, if $m_1(H) = 2$ and the two vertices of degree $1$ are non-adjacent, then $\brightness{H} \geq \binom{m(H)}{2}^{-1}$.
\end{claim}
\begin{proof}
    Since $H$ has $\ell \geq 2 > 0$ edges, we have $m(H) \geq 2$. Consider a uniformly random labeling $v_1, \dots, v_k$ of $V(H)$ and define the event
    \begin{equation*}
        \A \coloneqq \{\text{the last two non-isolated terms among $v_1, \dots, v_k$ are of degree $1$ and non-adjacent}\}.
    \end{equation*}
    Let us assume that $\A$ holds. Hence, the last two non-isolated terms $v_{i_1}, v_{i_2}$, $i_1 < i_2$, among $v_1, \dots, v_k$ are detectable by \cref{claim:characterization-of-detectable-vertices}. Since $v_{i_1}, v_{i_2}$ are non-adjacent and non-isolated, $v_{i_1}, v_{i_2}$ are both adjacent to a vertex that appears before them in the sequence $v_1, \dots, v_k$ and hence, they are the last two such vertices. Since $v_{i_1}, v_{i_2}$ are detectable, we conclude that $v_1, \dots, v_k$ is bright.
    We obtain $\A \subseteq \{v_1, \dots, v_k~\text{is bright}\}$. Hence, the probability that $v_1, \dots, v_k$ is bright, which is given by $\nu(H)$, is lower bounded by $\IP[A]$.
    For any vertex $v$ of degree $1$, there are at least $m_1(H) - 2$ vertices $w$ of degree $1$ such that $v \neq w$ and $v$ and $w$ are non-adjacent. Thus,
    \begin{equation*}
        \brightness{H} \geq \IP[\A] \geq \left(\frac{m_1(H)(m_1(H)-2)}{2}\right) \binom{m(H)}{2}^{-1}.
    \end{equation*}
    Furthermore, if $m_1(H) = 2$ and the two vertices of degree $1$ are non-adjacent, we have
    \begin{equation*}
        \brightness{H} \geq \IP[\A] = \binom{m(H)}{2}^{-1}.\qedhere
    \end{equation*}
\end{proof}

\begin{claim}\label{claim:beta-bound}
    We have $\brightness{H} \geq \frac{1}{12}$.
\end{claim}
\begin{proof}
    If $m_1(H) \geq 3$, we have $m_1(H) - 2 \geq \frac{m_1(H)}{3}$. Therefore, in the case $m_1(H) \geq \max\left\{3,\frac{m(H)}{2}\right\}$, \cref{claim:beta-bound-by-m1} implies
    \begin{equation*}
        \brightness{H} \geq \frac{1}{3} \left(\frac{m_1(H)}{m(H)}\right)^2 \geq \frac{1}{12}.
    \end{equation*}
    Now, we consider the case that $m_1(H) < \frac{m(H)}{2}$. Hence, $m_{\geq 2}(H) = m(H) - m_1(H) > \frac{m(H)}{2}$. In particular, $m_{\geq 2}(H) > \frac{m(H)}{2} \geq 1$, where we used that $m(H) \geq 2$ by $\ell \geq 2 > 0$. Since $m_{\geq 2}(H)$ is integral, we obtain $m_{\geq 2}(H) \geq 2$, which implies $m_{\geq 2}(H) - 1 \geq \frac{m_{\geq 2}(H)}{2}$. Therefore, by \cref{claim:beta-bound-by-m2},
    \begin{equation*}
        \brightness{H} \geq \frac{1}{2} \left(\frac{m_{\geq 2}(H)}{m(H)}\right)^2 > \frac{1}{8}.
    \end{equation*}
    Lastly, we consider the remaining case that $2 \geq m_1(H) \geq \frac{m(H)}{2}$.
    For $m_1(H) \leq 1$, it follows that $m(H) \leq 2m_1(H) \leq 2$, which is impossible since $\ell \geq 2$ implies $m(H) > 2$.
    For $m_1(H) = 2$, we have $m(H) \leq 2m(H) = 4$.
    From $m_1(H) = 2$, $m(H) \leq 4$, and $\ell \geq 2$, it follows that the two degree-$1$ vertices are non-adjacent. Therefore, \cref{claim:beta-bound-by-m1} with $m(H) \leq 4$ yields
    \begin{equation*}
        \brightness{H} \geq \binom{m(H)}{2}^{-1} \geq \frac{1}{6}.\qedhere
    \end{equation*}
\end{proof}

By \cref{lemma:ind-if-sparse,claim:beta-bound}, we finally obtain
\begin{equation*}
    \ind(H) \leq \frac{2 + 3e^2 \alpha}{2 + (e-2)\brightness{H}} \cdot \frac{1}{e} + 2\alpha \leq \frac{2 + 3e^2 \alpha}{2 + (e-2)\frac{1}{12}} \cdot \frac{1}{e} + 2\alpha = c.
\end{equation*}
The desired conclusion follows by $c \in \left[\frac{2}{e^2}, \frac{1}{e}\right)$.

\section{Proof of Lemma~\ref{lemma:ind-if-sparse}}\label{section:proof-of-lemma:ind-if-sparse}

Since $H$ has $k$ vertices and $\ell \geq 2$ edges, it follows that $k \geq 3$.
Let $n \geq k$ and let $G$ be a graph on $n$ vertices. Let $\vb V = (v_1, v_2, \dots) \in V(G)^\IN$ be uniformly random. For all positive integers $j$, we define $G_j \coloneqq G[\{v_1, \dots, v_j\}]$, and the event
\begin{equation*}
    \E_j \coloneqq \{v_1, \dots, v_j~\text{distinct}\} \cap \{M(G_j) \cong M(H)\}.
\end{equation*}
In particular, $\E_k = \{G_k \cong H\}$. Therefore, by symmetry,
\begin{equation}
    \indd{H}{G} = \IP[\E_k \mid v_1, \dots, v_k~\text{distinct}] = \IP[\E_k] + o_{n \mid k}(1), \label{equation:induced-density-by-p-of-e-k}
\end{equation}
where we used $\IP[v_1, \dots, v_k~\text{not distinct}] \leq o_{n \mid k}(1)$. For $w_1, \dots, w_k \in V(G)$ with $G[\{w_1, \dots, w_k\}] \cong H$, the distribution of $(v_1, \dots, v_k)$ conditioned on $\{v_1, \dots, v_k\} = \{w_1, \dots, w_k\}$ is a uniformly random labeling of $\{w_1, \dots, w_k\}$ by symmetry. Therefore, if $\IP[\E_k] > 0$, then, conditioned on $\E_k$, the labeled graph $(G_k, v_1, \dots, v_k)$ has the distribution of $(H, u_1, \dots, u_k)$, where $u_1, \dots, u_k$ is a uniformly random labeling of $V(H)$. We can use this to upper-bound $\IP[\E_k \backslash \E]$, where we define, by $k \geq 3$,
\begin{equation*}
    \E \coloneqq \E_{k-2} \cap \E_k = \E_{k-2} \cap \E_{k-1} \cap \E_k.
\end{equation*}
If $\IP[\E_k] = 0$, then $\IP[\E_k \backslash \E] = 0$. Otherwise, by conditioning on $\{v_1, \dots, v_k\}$,
\begin{equation*}
    \IP[\E_k \backslash \E] \leq \IP[\E^c \mid \E_k] = 1 - \binom{k-m(H)}{2}\binom{k}{2}^{-1} = \frac{-m(H)^2 + m(H)(2k-1)}{k(k-1)} \leq \frac{2m(H)}{k} \leq 2\alpha,
\end{equation*}
where we used $m(H) \leq \alpha k$ and $-m(H)^2 \leq -m(H)$ since $m(H) \geq 1$ by $\ell \geq 2 > 0$. Therefore, by \cref{equation:induced-density-by-p-of-e-k},
\begin{equation}
    \indd{H}{G} = \IP[\E_k] + o_{n \mid k}(1) = \IP[\E] + \IP[\E_k \backslash \E] + o_{n \mid k}(1) \leq \IP[\E] + 2\alpha + o_{n \mid k}(1). \label{equation:induced-density-by-p-of-e}
\end{equation}
Now, we will upper-bound $\IP[\E]$. Without loss of generality, we assume $\IP[\E] > 0$. We define a rule that, for any possible value of $\vb V$, assigns one of the colors black, green or red to each term $v_i$ of the sequence $\vb V$.
\begin{definition}
    For a positive integer $i$, we color $v_i$ \textit{black} if and only if neither of the following properties are satisfied by $\hat G \coloneqq M(G[\{v_j \mid 1 \leq j < i, v_j~\text{is black}\} \cup \{v_i\}])$:
    \begin{itemize}
        \item[(i)] There exists a vertex $v' \in V(H)$ such that $v'$ is detectable in $H$ and $\hat G \cong M(H - v')$,
        \item[(ii)] $\hat G \cong M(H)$.
    \end{itemize}
    Define $L \coloneqq \min\{i \geq 1 \mid \text{there are $k-2$ black terms among $v_1, \dots, v_i$}\}$ with $\min \emptyset \coloneqq \infty$. Denote the subsequence of black terms in $(v_i \mid 1 \leq i \leq L)$ by $\vb U$ and let $U$ be the set of sequence members of $\vb U$. For a positive integer $i$, if $v_i$ is not black and $M(G[U \cup \{v_i\}]) \cong M(H)$, we color $v_i$ \textit{red}. If $v_i$ is not black and not red, we color $v_i$ \textit{green}.
\end{definition}
\begin{fact}\label{fact:red-green-rules}
    For any integer $i \geq 1$,
    \begin{itemize}
        \item[\namedlabel{item:rule1-red-green-vertex-sequence}{(1)}] the indicator random variable $\one_{\{v_i~\text{is black}\}}$ is a function of $v_i$ and the subsequence of black terms in $(v_1, \dots, v_{i-1})$,
        \item[\namedlabel{item:rule2-red-green-vertex-sequence}{(2)}] the indicator random variables $\one_{\{v_i~\text{is red}\}}$ and $\one_{\{v_i~\text{is green}\}}$ are functions of $\one_{\{v_i~\text{is black}\}}$, $v_i$ and $\vb U$.
    \end{itemize}
\end{fact}

\begin{claim}\label{claim:isolated-is-black}
    For a positive integer $i$, if $v_i$ is isolated in $G_i$, then $v_i$ is black.
\end{claim}
\begin{proof}
    We first consider the case that $v_{i'}$ is black for some $1 \leq i' < i$ and choose $i'$ maximal with that property. Since $v_i$ is isolated in $G_i$, it is isolated in $G[\{v_j \mid 1 \leq j < i, v_j~\text{is black}\} \cup \{v_i\}]$. Hence,
    \begin{equation}\label{equation:prev-black}
        M(G[\{v_j \mid 1 \leq j < i, v_j~\text{is black}\} \cup \{v_i\}])
        =M(G[\{v_j \mid 1 \leq j < i', v_j~\text{is black}\} \cup \{v_{i'}\}]).
    \end{equation}
    Since $v_{i'}$ was black, it follows that $v_i$ is black. Now, let us consider the case that no term $v_j$ with $1 \leq j < i$ is black. Therefore,
    \begin{equation*}
        M(G[\{v_j \mid 1 \leq j < i, v_j~\text{is black}\} \cup \{v_i\}]) = M(G[\{v_i\}]) = G[\emptyset].
    \end{equation*}
    By $\ell \geq 2 > 0$, we have $G[\emptyset] \not\cong M(H)$ and by $\ell \geq 2$ combined with \cref{fact:graph-minus-detectable-vertex}, any detectable vertex $v \in V(H)$ satisfies $M(H - v) \not\cong G[\emptyset]$. It follows from the definition that $v_i$ is black.
\end{proof}
By \cref{claim:isolated-is-black}, $v_1$ is black. Moreoever, for an integer $i \geq 2$, if the maximal positive integer $i' < i$, such that $v_{i'}$ is black, satisfies $v_{i'} = v_i$, then $v_i$ is black due to \cref{equation:prev-black}. We conclude that for any positive integer $i$, we have $\IP[v_i~\text{is black}] = \IE[\IP[v_i~\text{is black} \mid v_1, \dots, v_{i-1}]] \geq \frac{1}{n}$ and thus, $\IP[\vb U~\text{has $k-2$ terms}] = \IP[L < \infty] = 1$.

Let $(v_{i_1}, \dots, v_{i_r})$ be the subsequence of $(v_1, \dots, v_k)$ consisting of all $v_i$ that are non-isolated in $G_i$. If $\E$ holds, then by $\E_k$, the graph $G_k \cong H$ has $\ell \geq 2 > 0$ edges, which implies $r \geq 1$.

\begin{claim}\label{claim:small-index-is-black}
    Suppose that $\E$ holds and $r \geq 2$. For every positive integer $i < i_{r-1}$, the term $v_i$ is black.
\end{claim}
\begin{proof}
    By $\E_k$, the vertices $v_1, \dots, v_k$ are distinct. Since $i < i_{r-1} < i_r \leq k$,
    \begin{equation*}
        \hat G \coloneqq M(G[\{v_j \mid 1 \leq j < i, v_j~\text{is black}\} \cup \{v_i\}])
    \end{equation*}
    is an induced subgraph of $\bar G \coloneqq M(G_k - v_{i_r} - v_{i_{r-1}})$.
    It follows from $i_r \leq k$ and since $v_{i_r}$ is non-isolated in $G_{i_r}$ that $v_{i_r}$ is non-isolated in $G_k$. Hence, $\bar G$ is a proper subgraph of $M(G_k) \cong M(H)$, where we used $\E_k$. Since $\hat G$ is a subgraph of $\bar G$, we conclude that $\hat G \not \cong M(H)$.
    Moreover, since $i_{r-1} < i_r$ and $v_{i_{r-1}}$ is non-isolated in $G_{i_{r-1}}$, it is non-isolated in $G_k - v_{i_r}$.
    Suppose that $\hat v \in V(G_k)$ satisfies $M(G_k - \hat v) \cong \hat G$. Since $\hat G$ is an induced subgraph of $\bar G$, which is an induced subgraph of $G_k - v_{i_r} - v_{i_{r-1}}$ and since $\degree[][G_k]{v_{i_r}} > 0$ and $\degree[][G_k - v_{i_r}]{v_{i_{r-1}}} > 0$, it follows that $\hat v$ is obscure and therefore not detectable in $G_k$, where we used the definitions of obscurity and detectability. Since $G_k \cong H$ by $\E_k$, this shows that there is no detectable vertex $v'$ in $H$ such that $M(H - v') \cong \hat G$.
    Therefore, $v_i$ is black by definition.
\end{proof}

The previous two claims imply:

\begin{corollary}\label{corollary:characterization-of-non-black-vertices}
    Suppose that $\E$ holds. For a positive integer $i \leq k$, $v_i$ is black, unless $i = i_r$, or $r \geq 2$ and $i = i_{r-1}$.
\end{corollary}

We define $Y$ and $Z$ to be the number of green and red terms in $(v_i \mid 1 \leq i \leq L)$, respectively, and consider the events $\A_1 \coloneqq \{(Y, Z) = (2, 0)\}$, $\A_2 \coloneqq \{(Y, Z) = (0, 1)\}$.

\begin{claim}\label{claim:e-subset}
    We have
    \begin{align*}
         & \E
        \subseteq \A_1 \cup \A_2,                                                    \\
         & \E \cap \{r \geq 2, v_{i_{r-1}}, v_{i_r}~\text{are detectable in $G_k$}\}
        \subseteq \A_1.
    \end{align*}
\end{claim}
\begin{proof}
    Let us assume that $\E$ holds. In particular, $\E_k$ implies that $v_1, \dots, v_k$ are distinct. By \cref{corollary:characterization-of-non-black-vertices}, there are at least $k-2$ black terms among $v_1, \dots, v_k$. By $\E_{k-2} \cap \E_k$, we have $M(G_{k-2}) \cong M(H) \cong M(G_k)$, which implies that $v_k$, $v_{k-1}$ are isolated in $G_k$. This yields $i_r \leq k-2$ and hence, by \cref{corollary:characterization-of-non-black-vertices}, $v_k, v_{k-1}$ are black. If all terms before $v_{i_r}$ are black (i.e. $r = 1$, or $r \geq 2$ and $v_{i_{r-1}}$ is black), then by $\E_k$,
    \begin{equation*}
        M(G[\{v_j \mid 1 \leq j < {i_r}, v_j~\text{is black}\} \cup \{v_{i_r}\}]) = M(G_k) \cong M(H),
    \end{equation*}
    implying that $v_{i_r}$ is not black. Hence, at least one term among $v_1, \dots, v_{k-2}$ is not black. By \cref{corollary:characterization-of-non-black-vertices}, the only potential non-black terms among $v_1, \dots, v_{k-2}$ are $\{i_s \mid \max\{1, r-1\} \leq s \leq r\}$ and thus, there are one or two such terms. Since $v_{k-1}, v_k$ are black, it follows that $L \in \{k-1, k\}$ and the set $I \coloneqq \{1 \leq i \leq L \mid v_i~\text{is not black}\}$ satisfies
    \begin{equation}\label{equation:size-of-i-given-e}
        \E \subseteq \{\emptyset \neq I \subseteq \{i_s \mid \max\{1, r-1\} \leq s \leq r\}\}
    \end{equation}
    Suppose that $|I| = 2$. Then, $L = k-2 + |I| = k$. For every $i \in I$, we claim that the term $v_i$ is green. Indeed, let $i \in I$. Then, there exists a $j \in I \backslash \{i\}$. By $j \in I$, we have $j \leq L = k$ and $v_j$ is not black and therefore, it is non-isolated in $G_j$ by \cref{claim:isolated-is-black}. Hence, $M(G_k - v_j)$ is a proper subgraph of $M(G_k) \cong M(H)$ by $\E_k$. In particular, $M(G_k - v_j) \not\cong M(H)$.
    However, by $I = \{i\} \sqcup \{j\}$ and $L = k$, we have $M(G[U \cup \{v_i\}]) = M(G_k - v_{j}) \not\cong M(H)$. Hence, $v_i$ is not red by definition.
    Since $v_i$ is not black by $i \in I$, it follows that $v_i$ is green.
    Thus,
    \begin{equation}\label{equation:a1-given-e-and-i}
        \E \cap \{|I| = 2\} \subseteq \{(Y, Z) = (2, 0)\} = \A_1.
    \end{equation}
    Assume $|I| = 1$. Then, $L = k-2 + |I| = k-1$ and $I = \{i\}$ with $1 \leq i \leq L = k-1$. By $\E_{k-1}$,
    \begin{equation*}
        M(G[U \cup \{v_i\}]) = M(G_{k-1}) \cong M(H).
    \end{equation*}
    It follows from the definition that $v_i$ is red, where we used that $v_i$ is not black by $i \in I$.
    Thus,
    \begin{equation}\label{equation:a2-given-e-and-i}
        \E \cap \{|I| = 1\} \subseteq \{(Y, Z) = (0, 1)\} = \A_2.
    \end{equation}
    The first part of the claim follows from \cref{equation:size-of-i-given-e,equation:a1-given-e-and-i,equation:a2-given-e-and-i}.

    If $r \geq 2$ and $v_{i_r}$ is detectable in $G_k$, then $v_{i_{r-1}}$ is not black since
    \begin{equation*}
        M(G[\{v_j \mid 1 \leq j < {i_{r-1}}, v_j~\text{is black}\} \cup \{v_{i_{r-1}}\}]) = M(G_k - v_{i_r})
    \end{equation*}
    by $\E$ and \cref{equation:size-of-i-given-e}. Therefore, if, in addition, $v_{i_{r-1}}$ is detectable in $G_k$, then $v_{i_r}$ is not black since
    \begin{equation*}
        M(G[\{v_j \mid 1 \leq j < {i_{r}}, v_j~\text{is black}\} \cup \{v_{i_{r}}\}]) = M(G_k - v_{i_{r-1}})
    \end{equation*}
    by $\E$ and \cref{equation:size-of-i-given-e}. It follows by \cref{equation:size-of-i-given-e} that
    \begin{equation}\label{equation:size-of-i-given-bright-and-e}
        \E \cap \{r \geq 2, v_{i_{r-1}}, v_{i_r}~\text{are detectable in $G_k$}\} \subseteq \{r \geq 2, I = \{i_{r-1}, i_r\}\} \subseteq \{|I| = 2\}.
    \end{equation}
    The second part of the claim follows from \cref{equation:size-of-i-given-bright-and-e,equation:a1-given-e-and-i}.
\end{proof}

Moreover, we define $\B \coloneqq \{\text{two consecutive terms in $(v_i \mid 1 \leq i \leq L)$ are not black}\}$.

\begin{claim}\label{claim:properties-of-e}
    We have
    \begin{align*}
        \IP[\A_1 \mid \E] & \geq \brightness{H}, \\
        \IP[\B \mid \E]   & \leq 3\alpha.
    \end{align*}
\end{claim}
\begin{proof}
    By definition, $\B \cap \A_2 = \emptyset$.
    If $\A_1 \cap \E$ holds, then by \cref{corollary:characterization-of-non-black-vertices}, $r \geq 2$ and the non-black vertices among $v_1, \dots, v_k$ are $v_{i_{r-1}}, v_{i_r}$.
    If $r \geq 2$ and $v_{i_{r-1}}, v_{i_r}$ are consecutive, then they are indeed the last two non-isolated terms in $G_k$ among $v_1, \dots, v_k$ and thus, the event
    \begin{equation*}
        \F \coloneqq \{\text{the last two non-isolated terms in $G_k$ among $v_1, \dots, v_k$ are consecutive}\}
    \end{equation*}
    follows.
    Therefore,
    \begin{equation*}
        \B \cap \A_1 \cap \E \subseteq \{r \geq 2, i_r = i_{r-1} + 1\} \subseteq \F.
    \end{equation*}
    Using $\E \subseteq \A_1 \cup \A_2$ by \cref{claim:e-subset} and $\B \cap \A_2 = \emptyset$, we conclude that
    \begin{equation*}
        \B \cap \E \subseteq (\B \cap \A_1 \cap \E) \cup (\B \cap \A_2) \subseteq \F.
    \end{equation*}
    Let $W_0 \in \binom{V(G)}{k}$ satisfy $G[W_0] \cong H$ and define
    \begin{equation*}
        \D \coloneqq \D(W_0) \coloneqq \E \cap \{\{v_1, \dots, v_k\} = W_0\} = \{\{v_1, \dots, v_k\} = W_0, \text{and $v_{k-1}, v_k$ are isolated in $G[W_0]$}\}.
    \end{equation*}
    Since we assumed $\IP[\E] > 0$ and since $\IP[\D]$ does not depend on the choice of $W_0$, we also have $\IP[\D] > 0$. By symmetry, we obtain:
    \begin{fact}\label{fact:conditional-distribution-of-v1-to-vk}
        The distribution of $(v_1, \dots, v_k)$ conditioned on $\D$ is uniform on the set
        \begin{equation*}
            \left\{\text{$(w_1, \dots, w_k)$ labeling of $W_0$} \mid \text{$w_{k-1}, w_k$ isolated in $G[W_0]$}\right\}.
        \end{equation*}
    \end{fact}
    We have $\B \cap \D \subseteq \B \cap \E \subseteq \F$ by $\D \subseteq \E$. By \cref{fact:conditional-distribution-of-v1-to-vk} and since $k \geq m(H) + 2$ by $\IP[\E] > 0$, we obtain
    \begin{equation}\label{equation:p-of-f-given-d}
        \IP[\B \mid \D] \leq \IP[\F \mid \D] = \frac{\binom{k-3}{m(H)-1}}{\binom{k-2}{m(H)}} = \frac{m(H)}{k-2} \leq \frac{3m(H)}{k} \leq 3\alpha,
    \end{equation}
    where we used $k \geq 3$ and $m(H) \leq \alpha k$.
    From $\D \subseteq \E$ and \cref{claim:e-subset}, we obtain
    \begin{equation}\label{equation:d-and-bright-imply-a1}
        \D \cap \{r \geq 2, v_{i_{r-1}}, v_{i_r}~\text{are detectable in $G_k$}\} \subseteq \A_1.
    \end{equation}
    By \cref{fact:conditional-distribution-of-v1-to-vk}, the relative order of the non-isolated vertices in $G_k = G[W_0] \cong H$ is uniformly random, conditioned on $\D$. Therefore, by the definitions of $\brightness{H}$, $\D$, and $v_{i_1}, \dots, v_{i_r}$,
    \begin{equation}\label{equation:p-of-a1-given-d}
        \brightness{H} = \IP[r \geq 2, v_{i_{r-1}}, v_{i_r}~\text{are detectable in $G_k$} \mid \D] \leq \IP[\A_1 \mid \D],
    \end{equation}
    where we used \cref{equation:d-and-bright-imply-a1}. Since the events $\D(W_0)$ over all $W_0 \in \binom{V(G)}{k}$ with $G[W_0] \cong H$ partition $\E$, \cref{equation:p-of-f-given-d,equation:p-of-a1-given-d} also hold conditioned on $\E$ instead of on $\D$.
\end{proof}

\begin{claim}\label{claim:bound-for-colors-of-red-green-vertex-sequence}
    There exists a real number $\lambda \in [0, 1]$ such that
    \begin{align*}
        \IP[\A_1 \cap \B^c]
         & \leq \lambda \cdot \frac{2}{e^2},   \\
        \IP[\A_2]
         & \leq (1-\lambda) \cdot \frac{1}{e}.
    \end{align*}
\end{claim}
\begin{proof}
    Let $t$ be a nonnegative integer and $\vb u = (u_1, \dots, u_t) \in V(G)^t$ be a sequence of length $t$ such that $\IP[\vb U = \vb u] > 0$. Since $\IP[\vb U~\text{has $k-2$ terms}] = 1$, it follows that $t = k-2$. Recall that $u_1 = v_1$ since $v_1$ is black. For a positive integer $i$ with $i \leq k-3$, let $\vb V_i$ be the subsequence of $\vb V$ of terms after $u_i$ and before $u_{i+1}$. Define $S_i$ as the length of $\vb V_i$ and let $P_i$ be the number of green terms in $\vb V_i$.

    Since $v_1, v_2, \dots$ are i.i.d. uniformly random from $V(G)$, applying \ref{item:rule1-red-green-vertex-sequence} in \cref{fact:red-green-rules} iteratively for $i = 1, 2, \dots$ implies that the conditional distribution of $(S_1, \dots, S_{k-3} \mid \vb U = \vb u)$ is such that $S_1, \dots, S_{k-3}$ are independent conditioned on $\vb U = \vb u$ and that, for $1 \leq i \leq k - 3$, the conditional distribution of $(S_i \mid \vb U = \vb u)$ is $\Geom(b_i) \in \{0, 1, 2, \dots\}$, where $b_i$ is the conditional probability, given $\vb U = \vb u$, that the term following $u_i$ is black. Therefore, since $v_1, v_2, \dots$ are i.i.d. uniformly random from $V(G)$, applying \ref{item:rule2-red-green-vertex-sequence} in \cref{fact:red-green-rules} repeatedly implies that the conditional distribution of $(S_1, \dots, S_{k-3}, P_1, \dots, P_{k-3} \mid \vb U = \vb u)$ is such that $(S_1, P_1), \dots, (S_{k-3}, P_{k-3})$ are independent conditioned on $\vb U = \vb u$ and that, for $1 \leq i \leq k-3$, the conditional distribution of $(P_i \mid \vb U = \vb u, S_i)$ is $\Binom(S_i, \frac{p_i}{p_i+q_i})$ (with $\frac{0}{0} \coloneqq 0$), where $p_i$ and $q_i$ are the conditional probabilities, given $\vb U = \vb u$, that the term following $u_i$ is green and red, respectively.

    By $u_1 = v_1$, the event $\A_1 \cap \B^c$ equals the event that for exactly two $i \in \{1, 2, \dots, k-3\}$, we have $(S_i, P_i) = (1, 1)$ and otherwise, $(S_i, P_i) = (0, 0)$. Similarly, $\A_2$ is the event that for exactly one $i \in \{1, 2, \dots, k-3\}$, we have $(S_i, P_i) = (1, 0)$ and otherwise, $(S_i, P_i) = (0, 0)$.
    Therefore,
    \begin{align*}
        \IP[\A_1 \cap \B^c \mid \vb U = \vb u]
         & = \sum_{1 \leq i_1 < i_2 \leq k-3} p_{i_1} p_{i_2} \prod_{1 \leq j \leq k-3} (1 - p_j - q_j) \leq \sum_{1 \leq i_1 < i_2 \leq k-3} p_{i_1} p_{i_2} e^{- p_j - q_j} \leq \frac{C^2}{2} e^{-C-D}, \\
        \IP[\A_2 \mid \vb U = \vb u]
         & = \sum_{1 \leq i \leq k-3} q_i \prod_{1 \leq j \leq k-3} (1 - p_j - q_j) \leq \sum_{1 \leq i \leq k-3} q_i \prod_{1 \leq j \leq k-3} e^{- p_j - q_j} = D e^{-C-D},
    \end{align*}
    where $C \coloneqq \sum_{1 \leq i \leq k-3} p_i$ and $D \coloneqq \sum_{1 \leq i \leq k-3} q_i$. By \cref{lemma:bound-lin-comb-of-poly-times-negative-ext}, there is a real number $\lambda_{\vb u} \in [0, 1]$ with
    \begin{align*}
        \IP[\A_1 \cap \B^c \mid \vb U = \vb u]
         & \leq \lambda_{\vb u} \cdot \frac{2}{e^2},     \\
        \IP[\A_2 \mid \vb U = \vb u]
         & \leq (1 - \lambda_{\vb u}) \cdot \frac{1}{e}.
    \end{align*}
    Let $\sum_{\vb u}$ denote the sum over $\vb u$ with $\IP[\vb U = \vb u] > 0$. Hence, $\lambda \coloneqq \sum_{\vb u} \lambda_{\vb u} \IP[\vb U = \vb u] \in [0, 1]$ satisfies
    \begin{align*}
        \IP[\A_1 \cap \B^c]
         & = \sum_{\vb u} \IP[\vb U = \vb u] \cdot \IP[\A_1 \cap \B^c \mid \vb U = \vb u]
        \leq \sum_{\vb u} \lambda_{\vb u} \IP[\vb U = \vb u] \cdot \frac{2}{e^2} = \lambda \cdot \frac{2}{e^2}, \\
        \IP[\A_2]
         & = \sum_{\vb u} \IP[\vb U = \vb u] \cdot \IP[\A_2 \mid \vb U = \vb u]
        \leq \sum_{\vb u} (1-\lambda_{\vb u}) \IP[\vb U = \vb u] \cdot \frac{1}{e} = (1-\lambda) \cdot \frac{1}{e}.\qedhere
    \end{align*}
\end{proof}

Let $\lambda \in [0, 1]$ satisfy the conclusion of \cref{claim:bound-for-colors-of-red-green-vertex-sequence}. By \cref{claim:bound-for-colors-of-red-green-vertex-sequence,claim:properties-of-e}, we have
\begin{align}\label{equation:p-of-b-given-e}
    \IP[\A_1 \cap \E]
     & \leq \IP[\A_1 \cap \B^c \cap \E] + \IP[\B \cap E] \leq \IP[\A_1 \cap \B^c] + \IP[\B \mid \E] \leq \lambda \cdot \frac{2}{e^2} + 3\alpha, \\
    \IP[\A_2 \cap \E]
     & \leq \IP[\A_2] \leq (1-\lambda) \cdot \frac{1}{e}.
\end{align}
We denote $p \coloneqq \IP[\A_1 \mid \E]$. Since $\E \subseteq \A_1 \sqcup \A_2$ by \cref{claim:e-subset}, $\IP[\A_2 \mid \E] = 1 - \IP[\A_1 \mid \E] = 1 - p$. In the following, we define $\frac{x}{0} = +\infty$ for $x \in \IR$. Since at least one of $p$ and $1-p$ is non-zero, it follows that
\begin{equation*}
    \IP[\E] = \min\left\{\frac{\IP[\A_1 \cap \E]}{p}, \frac{\IP[\A_2 \cap \E]}{1-p}\right\}
    \leq \min\left\{\frac{\lambda \cdot \frac{2}{e^2} + 3\alpha}{p}, \frac{(1-\lambda) \cdot \frac{1}{e}}{1-p}\right\}.
\end{equation*}
Since the real number $\lambda$ maximizing the right-hand side satisfies $\lambda = \frac{ep - 3e^2(1-p) \alpha}{2+(e-2)p}$ (not necessarily in $[0, 1]$), we conclude that
\begin{align*}
    \IP[\E] \leq \frac{2 + 3e^2 \alpha}{2 + (e-2)p} \cdot \frac{1}{e}
    \leq \frac{2 + 3e^2 \alpha}{2 + (e-2)\brightness{H}} \cdot \frac{1}{e},
\end{align*}
where we used $p = \IP[\A_1 \mid \E] \geq \brightness{H}$ by \cref{claim:properties-of-e}.
Hence, by \cref{equation:induced-density-by-p-of-e},
\begin{equation*}
    \indd{H}{G} \leq \IP[\E] + 2\alpha + o_{n \mid k}(1) \leq \frac{2 + 3e^2 \alpha}{2 + (e-2)\brightness{H}} \cdot \frac{1}{e} + 2\alpha + o_{n \mid k}(1).
\end{equation*}
Since this bound also holds in the case $\IP[\E] = 0$, the desired conclusion follows by taking $n \to \infty$.

\section*{Acknowledgments}

I would like to thank Lisa Sauermann for introducing me to the research topic and her guidance throughout this project.

\printbibliography

\begin{appendices}

    \section{General anti-concentration inequalities}\label{section:straightforward-bounds}

    \begin{lemma}\label{lemma:poisson-bound-for-bin-distribution}
        Let $\hat s$ be a fixed positive integer. For positive integers $s$ and $k$ and a real number $p \in [0, 1]$, we have
        \begin{equation*}
            \binom{k}{s} p^s(1-p)^{k-s} \leq
            \begin{cases}
                \binom{k}{s} \left(\frac{s}{k}\right)^s \left(\frac{k-s}{k}\right)^{k-s} & \text{for \quad $1 \leq s \leq k -1$},            \\
                \frac{{\hat s}^{\hat s}}{{\hat s}!e^{\hat s}} + o_k(1)                   & \text{for \quad $\hat s \leq s \leq k - \hat s$}.
            \end{cases}
        \end{equation*}
    \end{lemma}
    \begin{proof}
        Consider a positive integer $k$.
        For the first part, suppose that $s$ is a positive integer with $1 \leq s \leq k-1$. Let us find a real number $q \in [0, 1]$ maximizing
        \begin{equation*}
            f_s(q) \coloneqq \binom{k}{s} q^s(1-q)^{k-s}.
        \end{equation*}
        We have $f_s(q) \geq 0$ with equality for $q \in \{0, 1\}$. For $q \in (0, 1)$, we have $\frac{f_s'(q)}{f_s(q)} = \frac{s}{q} - \frac{k-s}{1-q}$. Thus, $f_s'(q) = 0$ if and only if $(1-q)s = q(k-s)$, which is equivalent to $q = \frac{s}{k}$. So, $f_s$ attains its unique maximum at $q = \frac{s}{k}$. Hence,
        \begin{equation*}
            f_s(p) \leq f_s\left(\frac{s}{k}\right) = \binom{k}{s} \left(\frac{s}{k}\right)^s \left(\frac{k-s}{k}\right)^{k-s}.
        \end{equation*}
        We will now prove the second part. So, assume that $k \geq 2\hat s \geq \hat s + 1$.
        The previous equation implies
        \begin{equation*}
            f_{\hat s}\left(\frac{\hat s}{k}\right)
            \leq (1 + o_k(1)) \cdot \frac{k^{\hat s}}{{\hat s}!} \cdot \frac{{\hat s}^{\hat s}}{k^{\hat s}} \left(1 - \frac{{\hat s}}{k}\right)^k
            \leq (1 + o_k(1)) \cdot \frac{{\hat s}^{\hat s}}{{\hat s}!e^{\hat s}}
            \leq \frac{{\hat s}^{\hat s}}{{\hat s}!e^{\hat s}} + o_k(1),
        \end{equation*}
        where we used that $\hat s$ is fixed.
        By $f_s(\frac{s}{k}) = f_{k-s}(\frac{k-s}{k})$, it is enough to prove that $f_s(\frac{s}{k})$ is increasing for $s = 1, 2, \dots, \lfloor \frac{k}{2} \rfloor$.
        For every positive integer $s$, Bernoulli's inequality with $-\frac{1}{(s+1)^2} > -1$ implies
        \begin{equation*}
            \frac{(1 + \frac{1}{s+1})^{s+1}}{(1 + \frac{1}{s})^s} = \frac{s+1}{s} \left(1 - \frac{1}{(s+1)^2}\right)^{s+1} \geq \frac{s+1}{s} \left(1-\frac{s+1}{(s+1)^2}\right) = 1.
        \end{equation*}
        For any positive integer $s$ with $s + 1 \leq \frac{k}{2}$, we have $s < k - s - 1$ and thus,
        \begin{equation*}
            \frac{f_s(\frac{s}{k})}{f_{s+1}(\frac{s+1}{k})} = \frac{\frac{1}{k-s}s^s (k-s)^{k-s}}{\frac{1}{s+1}(s+1)^{s+1}(k-s-1)^{k-s-1}} = \frac{(1+\frac{1}{k-s-1})^{k-s-1}}{(1+\frac{1}{s})^s} \geq 1.\qedhere
        \end{equation*}
    \end{proof}

    \begin{lemma}\label{lemma:poisson-bound-for-hypergeometric-distribution}
        Let $\hat s$ be a fixed positive integer. Let $Z$ be a random variable with a hypergeometric distribution with sample size $k$ in a population of size $n$ containing $r$ successes. For $1 \leq s \leq k-1$, we have
        \begin{equation*}
            \IP[Z = s] = \binom{k}{s} \left(\frac{r}{n}\right)^s \left(\frac{n-r}{n}\right)^{k-s} + o_{n \mid k}(1).
        \end{equation*}
        Thus,
        \begin{equation*}
            \IP[Z = s] \leq
            \begin{cases}
                \binom{k}{s} \left(\frac{s}{k}\right)^s \left(\frac{k-s}{k}\right)^{k-s} + o_{n \mid k}(1) & \text{for \quad $1 \leq s \leq k -1$},            \\
                \frac{{\hat s}^{\hat s}}{{\hat s}!e^{\hat s}} + o_k(1) + o_{n \mid k}(1)                   & \text{for \quad $\hat s \leq s \leq k - \hat s$}.
            \end{cases}
        \end{equation*}
    \end{lemma}
    \begin{proof}
        Let $N$ be a set of size $n$ and let $R \subseteq N$ have size $r$. Let $X_1, \dots, X_k \in N$ be independent and uniformly random. The event $\B \coloneqq \bigcup_{1 \leq i < j \leq k} \{X_i = X_j\}$
        satisfies
        \begin{equation*}
            \IP[\B] \leq \sum_{1 \leq i < j \leq k} \IP[X_i = X_j] = \binom{k}{2} \cdot \frac{1}{n} \leq o_{n \mid k}(1).
        \end{equation*}
        By symmetry, the distribution of $|\{1 \leq i \leq k \mid X_i \in R\}|$ conditioned on $\B^c$ is hypergeometric with sample size $k$ in a population of size $|N| = n$ containing $|R| = r$ successes. Therefore,
        \begin{align*}
            \IP[Z = s]
             & = \IP[|\{1 \leq i \leq k \mid X_i \in R\}| = s \mid \B^c]                                     \\
             & = \IP[|\{1 \leq i \leq k \mid X_i \in R\}| = s] + o_{n \mid k}(1)                             \\
             & = \binom{k}{s} \left(\frac{r}{n}\right)^s \left(\frac{n-r}{n}\right)^{k-s} + o_{n \mid k}(1).
        \end{align*}
        The remaining inequalities now follow from \cref{lemma:poisson-bound-for-bin-distribution}.
    \end{proof}

    \begin{lemma}\label{lemma:poisson-bound-for-multi-hypergeometric-distribution}
        Consider a fixed nonnegative integer $f$ and a fixed positive integer $s$. Let $n$ and $k$ be positive integers with $n \geq k$. Consider a set $R$ with size $|R| = n$ and disjoint subsets $R_1, \dots, R_f \subseteq R$. Let $W \in \binom{R}{k}$ be uniformly random and define $Z_i \coloneqq |W \cap R_i|$ for all $1 \leq i \leq f$. We have
        \begin{equation*}
            \IP[Z_1 = \dots = Z_f = s] \leq \left(\frac{s^s}{s!e^s}\right)^f + o_k(1) + o_{n \mid k}(1).
        \end{equation*}
    \end{lemma}
    \begin{proof}
        Without loss of generality, we may assume that $k \geq s+1$ as $s$ is fixed. We prove the statement by induction on $f$. The statement is trivial for $f = 0$. Now, suppose $f \geq 1$. Without loss of generality, we may assume $|R_f| \leq \frac{1}{f} \sum_{i=1}^f |R_f| \leq \frac{n}{f} \leq \frac{n}{2}$. Therefore, $R' \coloneqq R \backslash R_f$ satisfies $\left|R'\right| = n - |R_f| \geq \frac{n}{2} \geq \Omega_n(n)$. If $\IP[Z_f = s] = 0$, the claim is trivial. Thus, let us assume $\IP[Z_f = s] > 0$. By symmetry, the distribution of $W$, conditioned on $Z_f = s$, is uniform on $\binom{R'}{k-s}$. Therefore, by the inductive assumption for $f-1$,
        \begin{equation*}
            \IP[Z_1 = \dots = Z_{f-1} = s \mid Z_f = s] \leq \left(\frac{s^s}{s!e^s}\right)^{f-1} + o_k(1) + o_{n \mid k}(1),
        \end{equation*}
        where we used that $k-s = k - O(1)$ and $\left|R'\right| \geq \Omega_n(n)$. By \cref{lemma:poisson-bound-for-hypergeometric-distribution} with $s$ fixed and $1 \leq s \leq k-1$,
        \begin{equation*}
            \IP[Z_f = s] \leq \frac{s^s}{s!e^s} + o_k(1) + o_{n \mid k}(1).
        \end{equation*}
        By multiplying the two inequalities, we obtain
        \begin{equation*}
            \IP[Z_1 = \dots = Z_f = s] \leq \frac{s^s}{s!e^s}\left(\frac{s^s}{s!e^s}\right)^{f-1} + o_k(1) + o_{n \mid k}(1).\qedhere
        \end{equation*}
    \end{proof}

    \begin{lemma}\label{lemma:poisson-point-probability-decreasing}
        The function $\phi : s \mapsto \frac{s^s}{s!e^s}$ defined on the positive integers is monotonically decreasing and satisfies $\phi(s) \leq o_s(1)$.
    \end{lemma}
    \begin{proof}
        For $s \geq 1$, we have
        \begin{equation*}
            \left(\frac{s^s}{s!e^s}\right)\left(\frac{(s+1)^{s+1}}{{(s+1)!e^{s+1}}}\right)^{-1} = e \left(\frac{s}{s+1}\right)^s = \frac{e}{\left(1 + \frac{1}{s}\right)^s} \geq 1.
        \end{equation*}
        By Stirling's formula (see, for example, \cite{robbins:1955}),
        \begin{equation*}
            \frac{s^s}{s!e^s} = \frac{s^s}{(1+o_s(1))\sqrt{2\pi s}\left(\frac{s}{e}\right)^s e^s} = \frac{1 + o_s(1)}{\sqrt{2\pi s}} \leq o_s(1).\qedhere
        \end{equation*}
    \end{proof}

    \begin{lemma}\label{lemma:bound-poly-times-negative-exp}
        For all positive integers $s$ and nonnegative real numbers $x$, we have
        \begin{equation*}
            x^s e^{-x} \leq \left(\frac{s}{e}\right)^s.
        \end{equation*}
    \end{lemma}
    \begin{proof}
        For any real number $y$, we have $y + 1 \leq e^y$. Since $s \geq 1$, applying this bound for $y = \frac{x}{s} - 1$ yields
        \begin{equation*}
            \frac{x}{s} \leq e^{x / s - 1}.
        \end{equation*}
        Therefore, since $s$ is a positive integer and $x \geq 0$,
        \begin{equation*}
            x^s \leq s^s e^{x-s}.
        \end{equation*}
        The desired inequality follows by multiplying both sides by $e^{-x} > 0$.
    \end{proof}

    \begin{lemma}\label{lemma:bound-lin-comb-of-poly-times-negative-ext}
        For all nonnegative real numbers $y, z \geq 0$, there exists a real number $\lambda \in [0, 1]$ such that
        \begin{align*}
            \frac{y^2}{2} \cdot e^{-y-z} & \leq \lambda \cdot \frac{2}{e^2},   \\
            z e^{-y-z}                   & \leq (1-\lambda) \cdot \frac{1}{e}.
        \end{align*}
    \end{lemma}
    \begin{proof}
        The claim is equivalent to finding a real number $\lambda \in [0, 1]$ such that
        \begin{equation*}
            \frac{y^2 e^{2-y-z}}{4} \leq \lambda \leq 1 - z e^{1-y-z}.
        \end{equation*}
        So, it is sufficient to prove that for all nonnegative real numbers $y, z$, we have
        \begin{equation*}
            f(y, z) \coloneqq e^{y+z} - \frac{y^2 e^2}{4} -  ze \geq 0.
        \end{equation*}
        It is not hard to see that $f(y, z)$ attains a global minimum for some $(y^\ast, z^\ast) \geq 0$.
        If $y^\ast = 0$ or $z^\ast = 0$, the inequality is true by \cref{lemma:bound-poly-times-negative-exp}. Now, assume $y^\ast > 0$ and $z^\ast > 0$. Therefore, the partial derivatives of $f(y, z)$ with respect to $y$ and $z$ are $0$ at $(y^\ast, z^\ast)$. By differentiating with respect to $z$, we obtain $e^{y^\ast + z^\ast} - e = 0$ and thus, $y^\ast + z^\ast = 1$. Hence, by differentiating with respect to $y$, we obtain $e^{y^\ast + z^\ast} - \frac{y^\ast e^2}{2} = 0$, which yields $y^\ast = 2 e^{y^\ast + z^\ast - 2} = \frac{2}{e}$, and $z^\ast = 1 - y^\ast = 1 - \frac{2}{e}$. Therefore,
        \begin{equation*}
            f(y^\ast, z^\ast) = e - 1 - \left(1 - \frac{2}{e}\right)e = 1 \geq 0.\qedhere
        \end{equation*}
    \end{proof}
\end{appendices}

\end{document}